\documentclass[12pt,reqno]{amsart}
\usepackage[a4paper]{geometry}

\usepackage{amssymb,amsmath,amsthm,enumerate,bbm,epigraph}
\usepackage{graphicx}
\usepackage{color}

\DeclareMathOperator{\Dom}{Dom}
\DeclareMathOperator{\Ran}{Ran}
\DeclareMathOperator{\Ker}{Ker}
\DeclareMathOperator{\rank}{rank}

\DeclareMathOperator{\Tr}{Tr}
\DeclareMathOperator{\sech}{sech}
\DeclareMathOperator{\spec}{spec}

\renewcommand{\Re}{\operatorname{Re}}

\newcommand{\abs}[1]{\lvert#1\rvert}
\newcommand{\Abs}[1]{\left\lvert#1\right\rvert}
\newcommand{\norm}[1]{\lVert#1\rVert}

\newcommand{\jap}[1]{\langle#1\rangle}

\newcommand{\comp}{{\mathrm{comp}}}

\newcommand{\bbR}{{\mathbb R}}
\newcommand{\bbC}{{\mathbb C}}
\newcommand{\bbZ}{{\mathbb Z}}

\newcommand{\calB}{\mathcal{B}}

\newcommand{\calS}{\mathcal{S}}

\newcommand{\1}{\mathbf{1}}

\newcommand{\dd}{\mathrm d}
\newcommand{\ii}{\mathrm i}
\newcommand{\ee}{\mathrm e}

\numberwithin{equation}{section}

\theoremstyle{plain}
\newtheorem{theorem}{\bf Theorem}[section]
\newtheorem*{theorem*}{Theorem}
\newtheorem{lemma}[theorem]{\bf Lemma}
\newtheorem{proposition}[theorem]{\bf Proposition}
\newtheorem*{proposition*}{\bf Proposition}

\newtheorem{corollary}[theorem]{\bf Corollary}

\theoremstyle{definition}

\newtheorem*{definition*}{\bf Definition}

\theoremstyle{remark}
\newtheorem*{remark*}{\bf Remark}
\newtheorem{remark}[theorem]{\bf Remark}
\newtheorem{example}[theorem]{\bf Example}
\newtheorem*{example*}{\bf Example}

\newcommand{\eps}{\varepsilon}
\newcommand{\mySigma}{\Delta}

\begin{document}

\title[Eigenfunctions of positive Hankel operators]{Eigenfunctions of positive integral Hankel operators}

\author{Alexander Pushnitski}
\address{Department of Mathematics, King's College London, Strand, London, WC2R~2LS, U.K.}
\email{alexander.pushnitski@kcl.ac.uk}

\date{6 June 2026}

\begin{abstract}
We consider bounded positive semi-definite Hankel operators $H$, realised as integral operators on the positive semi-axis. For each value of $E$, not necessarily in the spectrum of $H$, we analyse solutions $f$ of the eigenvalue equation $Hf=Ef$, understood as an integral equation on the semi-axis. Our analysis reveals strong analogies with properties of solutions of the one-dimensional Schr\"odinger equation. 
\end{abstract}

\maketitle

\setlength{\epigraphwidth}{0.6\textwidth}
\epigraph{A mathematician is a person who can find analogies between theorems; a better mathematician is one who can see analogies between proofs and the best mathematician can notice analogies between theories.}{Stefan Banach}

\setcounter{tocdepth}{1}
\tableofcontents

\section{Introduction}\label{sec:a}

\subsection{Integral Hankel operators}
The main objects of interest in this paper are integral Hankel operators $H$, 
\begin{equation}
Hf(t)=\int_0^\infty h(t+s)f(s)\dd s, \quad t>0,
\label{eq:a4}
\end{equation}
viewed as linear operators on $L^2(\bbR_+)$. The function $h$ is known as the kernel function.

We will only be interested in positive (i.e. positive semi-definite) operators $H$. 
Positivity of $H$ is equivalent to a representation of the kernel function $h$ as the Laplace transform of a measure:
\begin{equation}
h(t)=\int_0^\infty \ee^{-t\lambda}\dd\mu(\lambda), \quad t>0.
\label{eq:a6}
\end{equation}
This fact is the continuous analogue of the solution to the classical Hamburger moment problem; see e.g. \cite[Theorems 5.1 and 5.3]{Ya2} for the precise statement and proof in a very general context. We will view the measure $\mu$ as the main functional parameter defining the operator $H$ and therefore we will sometimes write $h_\mu$ for the kernel function \eqref{eq:a6} and $H_\mu$ for $H$.  

The boundedness of $H_\mu$ is equivalent to the Carleson condition 
\begin{equation}
\mu([0,a))\leq C_\mu a, \quad a>0,
\label{eq:Carleson}
\end{equation}
see Proposition~\ref{prp:a1} below. We will assume \eqref{eq:Carleson} throughout the paper.

\begin{example}
The following basic example is of fundamental importance in the theory. 
Let $\mu$ be the Lebesgue measure on $\bbR_+$; then $h_\mu(t)=1/t$. The corresponding Hankel operator $H_\mu$ is known as the Carleman operator. 
The Carleman operator is bounded and self-adjoint.  Its spectrum is the interval $[0,\pi]$ and therefore its norm is $\norm{H_\mu}=\pi$, see e.g. \cite[Section 10.2]{Peller}.
\end{example}

\begin{example}
Let $\mu$ be a finite linear combination of point masses on $\bbR_+$ with positive coefficients. Then $H_\mu$ is a finite rank operator. 
\end{example}

\subsection{Hankel and Schr\"odinger}
It has been noticed in several places \cite{Howland,Ya3,PuSobolev} that the spectral properties of positive Hankel operators are in many ways analogous to the spectral properties of self-adjoint Schr\"odinger operators $S_v$ on the real line, 
\begin{equation}
S_vf(x)=-f''(x)+v(x)f(x)\quad\text{ in $L^2(\bbR)$.}
\label{eq:a6aa}
\end{equation}
To avoid technicalities, we shall always assume that the real-valued potential function $v$ is bounded on $\bbR$. Our purpose here is to extend this analogy from \emph{spectra} to  \emph{eigenfunctions}. 
More precisely, in the spectral theory of Schr\"odinger operator it is customary to consider  solutions $f$ to the Schr\"odinger equation
\begin{equation}
-f''(x)+v(x)f(x)=Ef(x), \quad x\in\bbR,
\label{eq:a8}
\end{equation}
where  $E$ may or may not be in the spectrum of $S_v$ and $f$ may or may not be in $L^2(\bbR)$. The behaviour of $f(x)$ at infinity is then related to the question of whether $E$ is in the spectrum of $S_v$. 

In the same way, in this paper we consider solutions to the integral equation 
\begin{equation}
\int_0^\infty h_\mu(t+s)f(s)\dd s=Ef(t), \quad t>0,
\label{eq:b7}
\end{equation}
where $E$ may or may not be in the spectrum of $H_\mu$ and $f$ may or may not be in $L^2(\bbR_+)$; we only require that the integral in \eqref{eq:b7} converges absolutely, so that the equation makes sense pointwise. 
By analogy with \eqref{eq:a8}, we will refer to \eqref{eq:b7} as the \emph{Hankel equation}.  

To avoid trivialities, we will always assume $H_\mu\not=0$ and we will only be interested in non-zero solutions to \eqref{eq:b7}. 

\subsection{General, regular and semi-regular cases}
Our main results are split into four groups, corresponding to different assumptions on the measure $\mu$. We view these four groups as analogous to problems for the Schr\"odinger equation
\begin{itemize}
\item
on the whole line $\bbR$ (general case);
\item
on $[0,\infty)$ and $(-\infty,0]$ (two semi-regular cases);
\item
on the finite interval $[0,1]$ (regular case).
\end{itemize}
We pause to describe our assumptions on the measure $\mu$ in each case. As already mentioned,  the measure $\mu$ is assumed to be Carleson \eqref{eq:Carleson} througout the paper. 
As usual, a measure $\mu$ is called \emph{finite}, if 
\begin{equation}
\int_0^\infty \dd\mu(\lambda)<\infty.
\label{eq:a10}
\end{equation}
Of course, for Carleson measures the finiteness condition is only needed at infinity, i.e. \eqref{eq:a10} is equivalent to 
\[
\int_1^\infty \dd\mu(\lambda)<\infty.
\]
We will call $\mu$ \emph{co-finite} (in the terminology of \cite{PuTreil}) if 
\begin{equation}
\int_0^\infty \frac{\dd\mu(\lambda)}{\lambda^2}<\infty.
\label{eq:a11}
\end{equation}
For Carleson measures, an elementary calculation with integrating by parts shows that co-finiteness is only needed at zero, i.e. \eqref{eq:a11} is equivalent to 
\[
\int_0^1 \frac{\dd\mu(\lambda)}{\lambda^2}<\infty.
\]

We will consider the following cases:
\begin{itemize}
\item
The \emph{general case} corresponds to \emph{all} Carleson measures. 
\item
In the \emph{semi-regular case}, one of the conditions \eqref{eq:a10}, \eqref{eq:a11} is assumed to hold, which corresponds to \emph{finite semi-regular} and the \emph{co-finite semi-regular} cases.
\item
In the \emph{regular case}, both \eqref{eq:a10} and \eqref{eq:a11} are assumed to hold.
\end{itemize}

We indicate the intended analogies between Schr\"odinger equation and Hankel equation in the table. 

{\renewcommand{\arraystretch}{1.3}  
\begin{center}
\begin{tabular}{|l|l|l|}
\hline
Hankel: terminology & Condition on $\mu$ & Schr\"odinger equation 
\\
\hline
General & Carleson & on $\bbR$
\\
Finite semi-regular & Carleson and finite & on $[0,\infty)$
\\
Co-finite semi-regular & Carleson and co-finite & on $(-\infty,0]$
\\
Regular & Both finite and co-finite & on $[0,1]$
\\
\hline
\end{tabular}
\end{center}
}
\begin{remark}
\begin{enumerate}[\rm (i)]
\item
In the regular case, the assumption that $\mu$ is both finite and co-finite implies that $\mu$ satisfies the Carleson condition:
\[
\int_0^a \dd\mu(\lambda)
\leq a\int_0^a \frac{\dd\mu(\lambda)}{\lambda}
\leq a\int_0^\infty \frac{\dd\mu(\lambda)}{\lambda}
\leq \frac{a}{2}\int_0^\infty \biggl(1+\frac1{\lambda^2}\biggr)\dd\mu(\lambda).
\]
\item
In the regular case $H_\mu$ is trace class and we have
\[
\Tr H_\mu=\int_0^\infty h_\mu(2t)\dd t
=\frac12\int_0^\infty \frac{\dd\mu(\lambda)}{\lambda}
<\infty.
\]
In particular, the spectrum of $H_\mu$ is discrete. 
\end{enumerate}
\end{remark}

\subsection{The kernel and $H_\mu^\perp$}
We recall that the kernel $\Ker H_\mu$ of a Hankel operator $H_\mu$ is either trivial or infinite-dimensional, see e.g.  \cite{KP}. It will be convenient to denote 
\[
H_\mu^\perp=H_\mu|_{(\Ker H_\mu)^\perp}. 
\]
To be precise, our analogy between Schr\"odinger operators and Hankel operators refers to $H_\mu^\perp$ rather than $H_\mu$:
\[
S_v\leftrightarrow H_\mu^\perp.
\]
We also recall \cite{KP} that $0$ is always in the spectrum of $H_\mu$, but may or may not be in the spectrum of $H_\mu^{\perp}$. 

\subsection{Spectral multiplicity}

To motivate our main results, we start from a discussion of spectral multiplicity; 
see e.g. \cite[Section VII.2]{RS1} for general background information. 

To set the scene, we recall the well-known facts concerning the Schr\"odinger operator $S_v$. 
\begin{proposition}
\label{prp:a00}
\begin{enumerate}[\rm (i)]
\item
Let $S_v$ be a Schr\"odinger operator \eqref{eq:a6aa} on $\bbR$. Then 
the multiplicity of the spectrum of $S_v$ is $\leq2$. 
\item
Let $S_v$ be a Schr\"odinger operator on the half-line $[0,\infty)$ (with any self-adjoint boundary condition at 0). Then $S_v$ has simple spectrum. 
\end{enumerate}
\end{proposition}
Both parts are well-known: (i) holds because the Schr\"odinger equation \eqref{eq:a8} is a differential equation of order two, which has two linearly independent solutions; (ii) holds because there is only one solution (up to scaling) satisfying the given boundary condition at zero.

The following proposition provides the precise analogy for Hankel operators. 
\begin{proposition}\cite{MPT,PuTreil}
\label{prp:a0}
Let $\mu$ be a Carleson measure on $(0,\infty)$ and let $H_\mu$ be the corresponding Hankel operator.
\begin{enumerate}[\rm (i)]
\item
The multiplicity of the spectrum of $H_\mu^{\perp}$ is $\leq2$. 
\item
Assume that $\mu$ is semi-regular (either finite or co-finite). Then $H_\mu^\perp$ has simple spectrum. 
\end{enumerate}
\end{proposition}

Part (i) is due to \cite{MPT} and part (ii) is due to \cite{PuTreil}. Note that the spectrum of the Carleman operator has multiplicity two, so in this respect part (i) is optimal. 

Proposition~\ref{prp:a0} is phrased in the abstract language of spectral multiplicity. 
One of the aims of this paper is to ``rephrase it'' in the language of the number of linearly independent solutions to the Hankel equation. But first we need to discuss the functional class where these solutions naturally belong.

\subsection{The class ${\calB}_0$}
We start with the key example that informs our intuition. 
\begin{example}
Let us discuss solutions to the eigenvalue equation for the Carleman operator $H_\mu$ (here $\dd\mu(\lambda)=\dd\lambda$).
Consider the identity
\[
\int_0^\infty \frac{s^{-\frac12-\kappa}}{s+t}\dd s
=
\pi\sec(\pi \kappa)t^{-\frac12-\kappa}, \quad t>0, \quad \kappa\in\bbC.
\]
Taking $\kappa=\pm\ii k$ with $k>0$, we obtain two solutions
\begin{equation}
f_{\pm \ii k}(t)=t^{-\frac12\pm\ii k}
\label{eq:a6a}
\end{equation}
corresponding to $E=\pi \sech(\pi k)\in (0,\pi)$, which is in the spectrum of $H_\mu$. On the other hand, taking $\kappa$ real in $(-1/2,1/2)$, we obtain the solutions
\begin{equation}
f_\kappa(k)=t^{-\frac12\pm\kappa}
\label{eq:a6b}
\end{equation}
corresponding to $E=\pi\sec(\pi\kappa)>\pi$, which is not in the spectrum of $H_\mu$. 
\end{example}

This example suggests that the borderline behaviour for a Hankel eigenfunction with $E\in\spec(H)$ is 
\[
\abs{f(t)}\sim t^{-\frac12},
\] 
both as $t\to\infty$ and as $t\to0$. 
This motivates the following definition.

\begin{definition*}
Let $f$ be a measurable function on $\bbR_+$ and let $\eps\in(0,\frac12)$. We will say that $f\in \calB_\eps$, if there exists $C_\eps>0$ such that
\begin{align*}
\abs{f(t)}\leq C_\eps t^{-\frac12-\eps} \text{ for } 0<t<1
\quad\text{ and }\quad
\abs{f(t)}\leq C_\eps t^{-\frac12+\eps} \text{ for } t>1.
\end{align*}
We will say that $f\in {\calB}_0$, if $f\in {\calB}_\eps$ for all $\eps>0$. 
\end{definition*}

For example, for any $\alpha>0$ functions of the form
\[
f(t)=(1+\abs{\log t})^\alpha t^{-\frac12}
\]
are in ${\calB}_0$. The functions \eqref{eq:a6a} are in ${\calB}_0$, while \eqref{eq:a6b} are not. 

We regard ${\calB}_0$ as the natural function class for solutions to the Hankel equation  with $E$ in the spectrum, similarly to the class of polynomially bounded functions for solutions of the Schr\"odinger equation.

\subsection{General case}

Our first main result is 
\begin{theorem}\label{thm:b5}
Let $H_\mu$ be a bounded positive Hankel operator. 
\begin{enumerate}[\rm (i)]
\item 
Let $H_\mu f=Ef$ where $f\in {\calB}_0$ and $E>0$. Then $E\in\spec(H_\mu)$.
\item
For $H_\mu$-spectrally almost every $E>0$, there exists $f\in {\calB}_0$ with $H_\mu f=Ef$. 
\item
For any $E>0$, there are at most two linearly independent $f\in {\calB}_0$ satisfying $H_\mu f=Ef$. 
\end{enumerate}
\end{theorem}

In the Schr\"odinger theory, the exact analogue of (i) holds with the class of polynomially bounded functions instead of ${\calB}_0$. This result is known as Shnol's theorem \cite{Shnol}, see also \cite[Section~C.4]{Simon} and \cite{Simon81}.

In (ii), ``$H$-spectrally'' means ``almost everywhere with respect to the spectral measure of $H$''. Again, in the Schr\"odinger theory, the exact analogue of (ii) holds for polynomially bounded functions (see \cite[Section~C.5]{Simon}). Simon dubs part (ii) the BGK theory, after the contributions of Berezanskii, Browder, G{\aa}rding, Gelfand and Kac. 

As already discussed, in the Schr\"odinger theory (iii) immediately follows from the fact that the Schr\"odinger equation is of order two. In the Hankel case it is far from trivial. Part (iii) is in agreement with Proposition~\ref{prp:a0}(i).

As in the Schr\"odinger case, parts (i) and (ii) of Theorem~\ref{thm:b5} immediately imply the following corollary (cf. \cite[Corollary~C.5.5]{Simon}). 
\begin{corollary}
The spectrum of $H_\mu^\perp$ is the closure of the set of $E>0$ where the Hankel equation \eqref{eq:b7} has a solution $f\in {\calB}_0$. 
\end{corollary}
\begin{proof}
Let $\mathcal E_{H_\mu}$ be the (projection-valued) spectral measure of $H_\mu$. If $E\in\spec(H_\mu)$, $E>0$, then $\mathcal E_{H_\mu}(E-a,E+a)\not=0$ for any $a>0$. Hence for $E'>0$ arbitrarily close to $E$ there are solutions $f\in {\calB}_0$ to the Hankel equation $H_\mu f=E'f$. 
\end{proof}

\begin{remark*}
In Theorem~\ref{thm:b5}(ii), in fact we prove a stronger statement, with solutions $f$ satisfying the logarithmic bound 
\[
\abs{f(t)}\leq C(1+\abs{\log t}) t^{-\frac12}.
\]
\end{remark*}

The proof of Theorem~\ref{thm:b5} is given in Sections~\ref{sec:d}--\ref{sec:f}.

\subsection{Semi-regular cases}
In Sections~\ref{sec:g} and \ref{sec:i} we study the finite and co-finite semi-regular cases. To keep this introduction concise, here we explain the nature of these results in general terms; precise formulations are provided at the beginning of the respective sections.

To set the scene, let us consider the Schr\"odinger operator $S_v$ on $[0,\infty)$ with the Dirichlet boundary condition at $x=0$. For every real number $E$, there exists a unique solution $\varphi_E$ of \eqref{eq:a8}, satisfying the boundary conditions 
\begin{align*}
\varphi_E(0)=0,\quad \varphi_E'(0)=1. 
\end{align*}
The fundamental result in the theory of the half-line spectral problem (see e.g. \cite[Chapter III]{Titchmarsh}) is that the integral transformation 
\[
f\mapsto \widetilde{f}(E)=\int_0^\infty \varphi_E(x)f(x)\dd x
\]
maps $L^2(\bbR_+)$ isometrically onto $L^2(\sigma)$ for some measure $\sigma$ on $\bbR$ ($\sigma$ is the \emph{spectral measure} of $S_v$), i.e. the \emph{Parseval relation} 
\[
\int_{-\infty}^\infty \abs{\widetilde{f}(E)}^2\dd\sigma(E)=\norm{f}^2
\]
holds. Moreover, this transformation \emph{diagonalises} $S_v$, i.e. 
\[
\widetilde{S_v f}(E)=E\widetilde{f}(E),
\quad
\text{$\sigma$-a.e. $E\in\bbR$.}
\]

In Section~\ref{sec:g} we prove an analogous result for Hankel operators in the \emph{finite semi-regular case}. Namely, for each $E>0$ we will define a function $\varphi_E$ on $\bbR_+$ such that
\begin{itemize}
\item
$\varphi_E(t)$ is regular at $t=0$ but may grow as $t\to\infty$;
\item
for each $E>0$, $\varphi_E$ is a \emph{weak solution} to the Hankel equation;
\item
for each $E>0$, if there is a solution $f\in {\calB}_0$ to $H_\mu f=Ef$, then $f$ is collinear to $\varphi_E$;
\item
the integral transformation 
\begin{equation}
f\mapsto \widetilde{f}(E)=\int_0^\infty \varphi_E(t)f(t)\dd t
\label{eq:a7}
\end{equation}
maps the closed range of $H_\mu$ isometrically onto $L^2(\sigma)$ with some finite measure $\sigma$ and diagonalises $H_\mu$. 
\end{itemize}
In the \emph{co-finite semi-regular} case in Section~\ref{sec:i} we define a solution $\theta_E$ that satisfies similar properties except that it is regular (in a suitable sense) at infinity rather than at zero. 

While the spectral problems on $(-\infty,0]$ and $[0,\infty)$ for the Schr\"odinger equation are trivially isomorphic by the reflection $x\mapsto -x$, we will see that there are some subtle differences between the finite semi-regular and co-finite semi-regular cases for the Hankel equation.

\subsection{Regular case}
In Section~\ref{sec:c} we study the regular case, which is perhaps the most interesting one from the standpoint of the underlying linear algebra. We outline the nature of these results here informally; for precise formulations, see the beginning of Section~\ref{sec:c}.

Let us start by considering the spectral problem for the Schr\"odinger equation on $[0,1]$ with the Dirichlet boundary conditions at both endpoints (i.e. the Sturm--Liouville problem). For every real number $E$, there are two solutions $\varphi_E$ and $\theta_E$ of the  Schr\"odinger equation \eqref{eq:a8}, satisfying the boundary conditions 
\begin{align*}
\varphi_E(0)=0,\quad  \varphi_E'(0)=1, 
\\
\theta_E(1)=0, \quad \theta_E'(1)=1. 
\end{align*}
If $E$ is an eigenvalue of $S_v$, then both $\varphi_E$ and $\theta_E$ are collinear with the corresponding eigenvector of the Schr\"odinger operator. If $E$ is not an eigenvalue, the solutions $\varphi_E$ and $\theta_E$ are linearly independent. The spectrum of $S_v$ can be described as a set of solutions
\begin{equation}
\text{ either of }\varphi_E(1)=0 \qquad \text{ or of } \theta_E(0)=0.
\label{eq:a7a}
\end{equation}

In Section~\ref{sec:c}, we will state and prove analogous results for Hankel operators in the regular case. That is, we will define two functions $\varphi_E$ and $\theta_E$ in $L^2(\bbR_+)$ such that 
\begin{itemize}
\item
both $\varphi_E$ and $\theta_E$ satisfy the Hankel equation up to a rank-one error term;
\item
the rank-one error term vanishes if and only if $E>0$ is an eigenvalue of $H$, and in this case $\varphi_E$ and $\theta_E$ are collinear with the corresponding eigenvector;
\item
the spectrum can be characterised in terms of rank-one conditions on $\varphi_E$ and $\theta_E$ in the spirit of \eqref{eq:a7a};
\item
the functions $\varphi_E$ and $\theta_E$ are linearly independent unless $E>0$ is an eigenvalue of $H$. 
\end{itemize}
The underlying linear algebra in the Hankel case is reminiscent of the Sturm-Liouville problem but has some refreshingly new aspects. In particular, the symmetrised Fredholm determinant
\[
\mySigma(E)=\frac{\det(I-\tfrac1E H)}{\det(I+\tfrac1E H)}
\]
plays an  important role in the argument.

\subsection{Commutation relations}
The fundamental reason behind the striking similarities between the theories of Hankel operators and Schr\"odinger operators remains a mystery (at least to the author). However, it is clear that the key role in the theory of Hankel operators is played by some rank-two and some rank-one commutation formulas known as \emph{Lyapunov equations} in linear systems theory, see e.g. \cite{Ober} or \cite[Chapter 12]{Peller}.
Namely, in the \emph{general case} we have 
\begin{equation}
B_2 H_\mu+H_\mu B_2=\text{ rank two operator }
\label{eq:a9}
\end{equation}
for some auxiliary operator $B_2$, while in each of the two \emph{semi-regular} cases (finite and co-finite), we have 
\begin{equation}
B_1 H_\mu+H_\mu B_1=\text{ rank one operator }
\label{eq:a9a}
\end{equation}
for some auxiliary operator $B_1$. In fact, the proof of Proposition~\ref{prp:a0} in \cite{MPT,PuTreil} is based precisely on these equations. In this paper, we use the Lyapunov equations (and similar proof strategy) to analyse the behaviour of solutions to the Hankel equation.

Specifically, in the context of this paper the rank-two equation \eqref{eq:a9} is realised as \eqref{eq:f10} and serves as a basis for the proof of Theorem~\ref{thm:b5}(iii). The rank-one equation \eqref{eq:a9a} appears in two variants, see \eqref{eq:c1a}, which underpin our analysis of the regular case. In the semi-regular cases, we use the construction of \cite{PuTreil} based on a related rank-one formula. 

To illustrate the utility of Lyapunov equations, let us explain how \eqref{eq:a9} is used in the proof of Theorem~\ref{thm:b5}(iii). Assume, to get a contradiction, that for some $E>0$ the subspace of solutions $f\in\calB_0$ to the Hankel equation $H_\mu f=Ef$ has dimension greater than two. Then there is a non-zero solution $f$ in this subspace such that the rank-two term in the right-hand side of \eqref{eq:a9} vanishes on $f$, which gives 
\[
H_\mu (B_2 f)=-E(B_2 f). 
\]
Thus $B_2f$ is an eigenvector corresponding to the negative eigenvalue $-E$. Since $H_\mu$ is a positive operator, it follows that $B_2 f=0$, which eventually forces $f=0$, yielding the required contradiction. 

\subsection{The structure of the paper}
Some general preliminary statements are given in Section~\ref{sec:b}.
In  Sections~\ref{sec:d}--\ref{sec:f} we prove Theorem~\ref{thm:b5}. 
The finite and co-finite semi-regular cases are discussed in Sections~\ref{sec:g} and \ref{sec:i}. The regular case is discussed in Section~\ref{sec:c}. 

The reader interested in the statements of our main results but not in the proofs may skip Sections~\ref{sec:b}--\ref{sec:f} and scan through the first parts of Sections~\ref{sec:g}--\ref{sec:c}.

\subsection{Acknowledgements} The author is grateful to Nikolai Filonov, Rupert Frank, Leonid Pastur and Alexander Sobolev for useful discussions.  

\section{Notation and preliminaries}\label{sec:b}

\subsection{Notation}
Throughout the paper, $\mu$ is a measure on $\bbR_+$ satisfying the Carleson condition \eqref{eq:Carleson}, $h=h_\mu$ is the Laplace transform \eqref{eq:a6} of $\mu$ and $H=H_\mu$ is the corresponding Hankel operator \eqref{eq:a4}. We skip the subscript $\mu$ if no confusion can arise.

We denote the inner product of $f$ and $g$ in $L^2(\bbR_+)$ by $\jap{f,g}$. Our inner product is linear in $f$ and anti-linear in $g$. Notation $\jap{\cdot,f}g$ stands for the rank-one operator $\psi\mapsto\jap{\psi,f}g$. 

We will consider the Hankel equation \eqref{eq:b7} for measurable functions $f$ satisfying 
\begin{equation}
\int_0^\infty \abs{f(t)}(t+1)^{-1}\dd t<\infty.
\label{eq:b7a}
\end{equation}
In this case the integral in \eqref{eq:b7} converges absolutely by the estimate \eqref{eq:h-estimate} below. We note in particular that any function $f\in {\calB}_\eps$ with $\eps\in(0,\frac12)$ satisfies \eqref{eq:b7a}. 

For a self-adjoint operator $A$ and for a Borel set $\Delta\subset\bbR$, we denote by ${\mathcal E}_A(\Delta)$ the spectral projection of $A$ corresponding to $\Delta$.

 Let $C^\infty_{\comp}(\bbR_+)$ be the set of infinitely smooth compactly supported functions $f$ on $\bbR_+$; in particular, the support of each $f\in C^\infty_{\comp}(\bbR_+)$ is separated away from the origin. Let $C^\infty_{\comp,0}(\bbR_+)\subset C^\infty_{\comp}(\bbR_+)$ be the subset of functions $f$ with \emph{zero average}, i.e. 
\begin{equation}
\int_0^\infty f(t)\dd t=0.
\label{eq:g2a}
\end{equation}
Equivalently, 
\[
C^\infty_{\comp,0}(\bbR_+)=\{g': g\in C^\infty_{\comp}(\bbR_+)\}.
\]
It is easy to see that $C^\infty_{\comp,0}(\bbR_+)$ is dense in $L^2(\bbR_+)$.

\subsection{The Laplace transform and factorisation of $H_\mu$}
We denote by $L_\mu$ the Laplace transform, considered as an operator from $L^2(\bbR_+)$ to $L^2(\mu):=L^2(\bbR_+,\dd\mu)$:
\begin{equation}
L_\mu: L^2(\bbR_+)\to L^2(\mu), \quad (L_\mu f)(x)=\int_{0}^{\infty}\ee^{-tx}f(t)\dd t, \quad x>0.
\label{eq:g2}
\end{equation}
The boundedness of $L_\mu$ follows from (in fact, equivalent to) the Carleson condition \eqref{eq:Carleson} on $\mu$, see Proposition~\ref{prp:a1} below. We also use the adjoint
\[
L_\mu^*: L^2(\mu)\to L^2(\bbR_+), \quad (L_\mu^{*}f)(t)=\int_{0}^{\infty}\ee^{-t\lambda}f(\lambda)\dd\mu(\lambda), \quad t>0.
\]
We have the important factorisation
\begin{equation}
\boxed{H_\mu=L_\mu^*L_\mu}
\label{eq:g3a}
\end{equation}
see \cite[Proposition~1.3]{PuTreil}.
In particular, $\Ker H_\mu=\Ker L_\mu$ and $\overline{\Ran H_\mu}=\overline{\Ran L_\mu^*}$.

\subsection{The boundedness of $H_\mu$}
The statement below can be regarded as folkore; see e.g.  \cite[Proposition~A.1]{PuTreil0}.
\begin{proposition}\label{prp:a1}
For a measure $\mu$ on $\bbR_+$, the following are equivalent:
\begin{enumerate}[\rm (i)]
\item
$\mu$ is Carleson \eqref{eq:Carleson};
\item
$h_\mu$ satisfies 
\begin{equation}
\abs{h_\mu(t)}\leq C_h/t, \quad t>0,
\label{eq:h-estimate}
\end{equation}
with some constant $C_h>0$;
\item
$H_\mu$ is bounded on $L^2(\bbR_+)$;
\item
$L_\mu: L^2(\bbR_+)\to L^2(\mu)$ is bounded. 
\end{enumerate}
\end{proposition}

In fact, the implication (i)$\Rightarrow$(ii) is a one-line argument \eqref{eq:b7c} used in the next proposition.

\subsection{Smoothness and estimates of $h_\mu$ and of solutions to the Hankel equation}

The following propositions are quite elementary and probably well-known. 

\begin{proposition}\label{thm:b4}
Let $\mu$ be a Carleson measure and let $h=h_\mu$, $H=H_\mu$. 
\begin{enumerate}[\rm (i)]
\item
The kernel function $h(t)$ extends to an analytic function in $\Re t>0$ and satisfies the estimates 
\begin{equation}
h(t)\leq \frac{C_\mu}{t}, \quad
\abs{h'(t)}\leq \frac{2C_\mu}{t^2}, \quad t>0. 
\label{eq:b7b}
\end{equation}
\item
Suppose $f$ satisfies \eqref{eq:b7a} and $f$ is a solution to the Hankel equation for some $E\not=0$. 
Then $f\in C^\infty(\bbR_+)$; moreover, $f(t)$ extends to an analytic function in the right half-plane $\Re t>0$ and satisfies the estimates
\begin{equation}
f(t)=o(1), \quad t\to\infty,
\qquad
f(t)=o(t^{-1}), \quad t\to0_+.
\label{eq:b7d}
\end{equation}
\end{enumerate}
\end{proposition}
\begin{proof}
(i) Analyticity of $h(t)$ in $\Re t>0$ is clear from the definition \eqref{eq:a6}. 
To prove the first estimate in  \eqref{eq:b7b}, we integrate by parts:
\begin{equation}
h(t)=\int_0^\infty \ee^{-t\lambda}\dd\mu(\lambda)
=t\int_0^\infty \ee^{-t\lambda}\mu(0,\lambda)\dd\lambda
\leq 
C_\mu t \int_0^\infty \ee^{-t\lambda}\lambda\dd\lambda
=
C_\mu/t,
\label{eq:b7c}
\end{equation}
and similarly for the second estimate
\begin{align}
-h'(t)&=\int_0^\infty \ee^{-t\lambda}\lambda \dd\mu(\lambda)
=t\int_0^\infty \ee^{-t\lambda}\lambda\mu(0,\lambda)\dd\lambda
-\int_0^\infty \ee^{-t\lambda}\mu(0,\lambda)\dd\lambda
\notag
\\
&\leq 
C_\mu t\int_0^\infty \ee^{-t\lambda}\lambda^2\dd\lambda
=
2C_\mu/t^2.
\label{eq:b7dd}
\end{align}

(ii) We have 
\[
f(t)=\frac1E\int_0^\infty h(t+s)f(s)\dd s=\frac1E\int_0^\infty \ee^{-\lambda t}(L_\mu f)(\lambda)\dd\mu(\lambda), \quad t>0,
\]
where $L_\mu f$ is the Laplace transform of $f$ \eqref{eq:g2}. From the second representation the analyticity of $f(t)$ is evident. From the first representation we find
\[
\abs{f(t)}\leq \frac{C_\mu}{\abs{E}}\int_0^\infty \frac{\abs{f(s)}}{s+t}\dd s
=
\frac{C_\mu}{\abs{E}}\int_0^\infty \frac{s+1}{s+t}\frac{\abs{f(s)}}{s+1}\dd s,
\]
and \eqref{eq:b7d} follows by dominated convergence. 
\end{proof}

\begin{proposition}\label{prp:b5}
\begin{enumerate}[\rm (i)]
\item
Suppose $\mu$ is finite. Then the kernel function satisfies
\begin{equation}
h(t)\leq h(0)= \mu(\bbR_+)
\quad\text{ and }\quad
\abs{h'(t)}\leq \frac{\mu(\bbR_+)}{t}.
\label{eq:b1}
\end{equation}
\item
Suppose $\mu$ is co-finite. Then the kernel function satisfies
\begin{equation}
h(t)\leq \frac{2}{t^2}\int_0^\infty \frac{\dd\mu(\lambda)}{\lambda^2}
\quad\text{ and }\quad
\abs{h'(t)}\leq \frac{6}{t^3}\int_0^\infty \frac{\dd\mu(\lambda)}{\lambda^2}.
\label{eq:b1a}
\end{equation}
\end{enumerate}
\end{proposition}
\begin{proof}
(i) The first relation \eqref{eq:b1} is evident. To prove the second relation, we use \eqref{eq:b7dd}:
\[
\abs{h'(t)}
\leq 
t\int_0^\infty \lambda \ee^{-t\lambda}\mu(0,\lambda)\dd\lambda
\leq 
t\mu(\bbR_+)
\int_0^\infty \lambda \ee^{-t\lambda}\dd\lambda
=
\frac{\mu(\bbR_+)}{t}.
\]
(ii) We estimate
\[
\mu(0,a)\leq a^2\int_0^a\frac{\dd\mu(\lambda)}{\lambda^2}\leq a^2 \int_0^\infty\frac{\dd\mu(\lambda)}{\lambda^2}
\]
and then use \eqref{eq:b7c}:
\begin{align*}
h(t)&=
t\int_0^\infty \ee^{-t\lambda}\mu(0,\lambda)\dd\lambda
\leq
t \int_0^\infty\frac{\dd\mu(\lambda)}{\lambda^2} 
\int_0^\infty \ee^{-t\lambda}\lambda^2\dd\lambda
=
\frac{2}{t^2}\int_0^\infty \frac{\dd\mu(\lambda)}{\lambda^2}.
\end{align*}
Similarly, using \eqref{eq:b7dd}
\[
\abs{h'(t)}
\leq 
t\int_0^\infty  \ee^{-t\lambda}\lambda\mu(0,\lambda)\dd\lambda
\leq
t\int_0^\infty\frac{\dd\mu(\lambda)}{\lambda^2} 
\int_0^\infty  \ee^{-t\lambda}\lambda^3\dd\lambda
=
\frac{6}{t^3}\int_0^\infty \frac{\dd\mu(\lambda)}{\lambda^2}. 
\]
\end{proof}
If $\mu$ is co-finite, we denote the ``integrated kernel'' by 
\[
h^{\rm int}(t):=\int_{t}^\infty h(s)\dd s=\int_0^\infty \ee^{-t\lambda}\frac{\dd\mu(\lambda)}{\lambda},
\]
where by \eqref{eq:b1a} 
\begin{equation}
h^{\rm int}(t)\leq \frac{2}{t} \int_0^\infty\frac{\dd\mu(\lambda)}{\lambda^2}. 
\label{eq:b1b}
\end{equation}

\section{Proof of Theorem~\ref{thm:b5}(i) (Shnol's theorem)}\label{sec:d}

\subsection{The main argument}
Let $W$ be the operator of multiplication by the weight function 
\[
w(t)=\frac{t}{t^2+1}, \quad  t>0
\]
in $L^2(\bbR_+)$. 
For a bounded positive Hankel operator $H=H_\mu$ and $\alpha\in(-1/2,1/2)$, we will consider the operator 
\[
M_\alpha=W^\alpha H W^{-\alpha}
\]
in $L^2(\bbR)$. This operator is well-defined on the dense space $C_{\comp}^\infty(\bbR_+)$. 
Indeed, for any $f\in C_{\comp}^\infty(\bbR_+)$, by \eqref{eq:h-estimate}, the function $HW^{-\alpha} f$ satisfies the estimate 
\[
\abs{(HW^{-\alpha} f)(t)}\leq C/(t+1), \quad t>0,
\]
and so the function $W^\alpha H W^{-\alpha}f$ is in $L^2(\bbR_+)$ for any $\alpha\in(-1/2,1/2)$. 

\begin{lemma}\label{lma:d1}
For any bounded positive Hankel operator $H$ and $\alpha\in(-1/2,1/2)$, the operator $M_\alpha$ extends to $L^2(\bbR)$ as a bounded operator. Moreover, $M_\alpha$ is continuous in $\alpha\in(-1/2,1/2)$ in the operator norm. 
\end{lemma}
We postpone the routine technical proof of this lemma to the end of this section. Here, assuming the lemma, we give the proof of ``Shnol's theorem'' for Hankel operators. 

\begin{proof}[Proof of Theorem~\ref{thm:b5}(i) (Shnol's theorem)]
Suppose that $Hf=Ef$ with $f\in {\calB}_0$. The assumption $f\in {\calB}_0$ implies that for any $\alpha>0$, we have $g=W^\alpha f\in L^2(\bbR_+)$. Multiplying $Hf=Ef$ by $W^{\alpha}$, we find 
\[
W^\alpha H W^{-\alpha}g=Eg, \quad g=W^\alpha f \in L^2(\bbR_+). 
\]
This means that $E$ is an eigenvalue of $W^\alpha H W^{-\alpha}$ for all $\alpha>0$. By Lemma~\ref{lma:d1} (the norm continuity of $W^\alpha H W^{-\alpha}$ at $\alpha=0$) it follows that $E$ is in the spectrum of $H$. The proof of Theorem~\ref{thm:b5}(i) is complete. 
\end{proof}

\subsection{Estimates for the weight  $w$}
In the rest of this section we prove Lemma~\ref{lma:d1}. We start with some elementary estimates for the weight $w$. Observe that 
\begin{equation}
w(t)=\frac{t}{t^2+1}=\frac1{t+t^{-1}}=\frac{1}{2\cosh x}=\frac12\sech x, 
\quad x=\log t.
\label{eq:d1a}
\end{equation}
\begin{lemma}\label{lma:d2}
For any $\eps\in\bbR$ and any $s,t>0$, we have
\begin{align}
w(ts)&>w(t)w(s), 
\label{eq:d8}
\\
\abs{(w(ts)/w(s))^\eps-1}
&\leq 
\abs{w(t)^{\eps}-w(t)^{-\eps}}. 
\label{eq:d10}
\end{align}
\end{lemma}
\begin{proof}
Writing $x=\log t$ and $y=\log s$ and using \eqref{eq:d1a}, the ratio $w(ts)/w(s)$ can be written as
\begin{align*}
\frac{w(ts)}{w(s)}
&=\frac{\sech(x+y)}{\sech y}
=\frac{\cosh y}{\cosh (x+y)}
\\
&=\frac{\cosh y}{\cosh x\cosh y +\sinh x\sinh y}
=\frac1{\cosh x+\sinh x\tanh y}.
\end{align*}
Since $\tanh y\in (-1,1)$, from here we find 
\begin{equation}
\ee^{-\abs{x}}<\frac{w(ts)}{w(s)}<\ee^{\abs{x}}.
\label{eq:d10a}
\end{equation}
Furthermore, from $w(t)^{-1}=\ee^{\abs{x}}+\ee^{-\abs{x}}$ it is immediate that $w(t)<\ee^{-\abs{x}}$. Combining this with \eqref{eq:d10a}, we obtain the $s$-independent bound
\[
w(t)<\frac{w(ts)}{w(s)}<w(t)^{-1},
\]
whence \eqref{eq:d8} and \eqref{eq:d10} easily follow. 
\end{proof}

\subsection{Proof of Lemma~\ref{lma:d1}}
Let us first prove that $M_\alpha$ extends to a bounded operator by giving a bound on its quadratic form. The argument below is a minor variation of the classical proof of the boundedness of the Carleman operator via the Schur test. 

Below $f,g\in C_{\comp}^\infty(\bbR_+)$. We have
\begin{align*}
\jap{M_\alpha f,g}
=
\jap{HW^{-\alpha}f,W^\alpha g}
=
\int_0^\infty \int_0^\infty h(t+s)(w(t)/w(s))^\alpha f(s)\overline{g(t)}\dd s\, \dd t. 
\end{align*}
Using \eqref{eq:h-estimate}, we find 
\begin{align*}
\abs{\jap{M_\alpha f,g}}
\leq
C_{h}
\int_0^\infty \int_0^\infty (w(t)/w(s))^\alpha \frac{\abs{f(s)g(t)}}{t+s}\dd s\, \dd t. 
\end{align*}
It suffices to consider the case $\alpha\in(0,1/2)$ (the case $\alpha\in(-1/2,0)$ is considered by interchanging the roles of $f$ and $g$). Let us fix some $\beta\in\bbR$ such that $0<\beta<1$ and $0<\beta+2\alpha<1$. We multiply and divide the integrand above by $(t/s)^{\beta/2}$ and use Cauchy-Schwarz. This yields
\[
\abs{\jap{M_\alpha f,g}}\leq C_*\sqrt{AB}, 
\]
where
\begin{align*}
A&=
\int_0^\infty \int_0^\infty (s/t)^\beta \frac{\abs{f(s)}^2}{t+s}\dd s\, \dd t, 
\\
B&=
\int_0^\infty \int_0^\infty (w(t)/w(s))^{2\alpha} (t/s)^\beta \frac{\abs{g(t)}^2}{t+s}\dd s\, \dd t.
\end{align*}
Let us rewrite $A$ as 
\[
\int_0^\infty \abs{f(s)}^2 \left\{\int_0^\infty (s/t)^\beta \frac{\dd t}{t+s}\right\}\dd s,
\]
and for the integral in $\{\cdots\}$ by a change of variable we find
\[
\int_0^\infty (s/t)^\beta \frac{\dd t}{t+s}=\int_0^\infty \frac{\dd t}{t^\beta(t+1)}=C_\beta
\]
which is independent of $s$. We thus find 
\[
A\leq C_\beta\norm{f}^2.
\]
In a similar way, for $B$ be have 
\[
B\leq \int_0^\infty\abs{g(t)}^2\left\{\int_0^\infty \left(\frac{w(t)}{w(s)}\right)^{2\alpha}\frac{t^\beta}{s^\beta}\frac{\dd s}{t+s}\right\}\dd t
\]
and for the integral in $\{\cdots\}$ by a change of variable we find
\[
\int_0^\infty \left(\frac{w(t)}{w(s)}\right)^{2\alpha}\frac{t^\beta}{s^\beta}\frac{\dd s}{t+s}
=
\int_0^\infty \left(\frac{w(t)}{w(ts)}\right)^{2\alpha}\frac{1}{s^\beta}\frac{\dd s}{1+s}.
\]
Using \eqref{eq:d8}, we estimate the last integral as 
\[
\int_0^\infty \left(\frac{w(t)}{w(ts)}\right)^{2\alpha}\frac{1}{s^\beta}\frac{\dd s}{1+s}
\leq
\int_0^\infty\frac{1}{w(s)^{2\alpha}s^\beta}\frac{\dd s}{1+s}=C_{\alpha,\beta}
\]
where the right hand side is finite and independent of $t$, which gives 
\[
B\leq C_{\alpha,\beta}\norm{g}^2.
\]
We have obtained the estimate
\[
\abs{\jap{M_\alpha f,g}}\leq C\norm{f}\norm{g},
\]
and so $M_\alpha$ is bounded.

Let us prove that $M_\alpha$ depends continuously on $\alpha$ in the operator norm. The estimates below are based on the same idea as above, but are somewhat more involved. We fix $\alpha\in[0,1/2)$ (the case $\alpha\in(-1/2,0]$ is considered in a symmetric fashion); let us prove that 
\[
\norm{M_{\alpha+\eps}-M_\alpha}\to0
\]
as $\eps\to0$. We will prove that for all $f$, $g$ as above, 
\[
\abs{\jap{(M_{\alpha+\eps}-M_\alpha)f,g}}\leq K_{\alpha,\eps}\norm{f}\norm{g}
\]
where $K_{\alpha,\eps}\to0$ as $\eps\to0$. Similarly to the first part of the proof we find
\begin{align*}
\abs{\jap{(M_{\alpha+\eps}-M_\alpha)f,g}}
&\leq
C_h\int_0^\infty\abs{(w(t)/w(s))^\eps-1}(w(t)/w(s))^\alpha \frac{\abs{f(s)g(t)}}{t+s}\dd s\, \dd t
\\
&\leq C_h \sqrt{A_\eps B_\eps},
\end{align*}
where 
\begin{align*}
A_\eps&=
\int_0^\infty \int_0^\infty \abs{(w(t)/w(s))^\eps-1} (s/t)^\beta \frac{\abs{f(s)}^2}{t+s}\dd s\, \dd t, 
\\
B_\eps&=
\int_0^\infty \int_0^\infty \abs{(w(t)/w(s))^\eps-1} (w(t)/w(s))^{2\alpha} (t/s)^\beta \frac{\abs{g(t)}^2}{t+s}\dd s\, \dd t.
\end{align*}
We have 
\begin{align*}
A_\eps=
\int_0^\infty \abs{f(s)}^2 a_\eps(s)\dd s, 
\quad
B_\eps=
\int_0^\infty \abs{g(t)}^2 b_\eps(t)\dd t,
\end{align*}
where
\begin{align*}
a_\eps(s)&=\int_0^\infty \abs{(w(t)/w(s))^\eps-1} (s/t)^\beta \frac{\dd t}{t+s},
\\
b_\eps(t)&=\int_0^\infty \abs{(w(t)/w(s))^\eps-1} (w(t)/w(s))^{2\alpha} (t/s)^\beta \frac{\dd s}{t+s}.
\end{align*}
Our aim is to prove that 
\begin{align}
\lim_{\eps\to0}\sup_{s>0} a_\eps(s)=0,
\label{eq:d5}
\\
\lim_{\eps\to0}\sup_{t>0} b_\eps(t)=0.
\label{eq:d6}
\end{align}
Changing variable in the integral for $a_\eps(s)$ and using \eqref{eq:d10}, we find
\begin{align*}
a_\eps(s)&=\int_0^\infty \abs{(w(ts)/w(s))^\eps-1} \frac{\dd t}{t^{\beta}(t+1)}
\\
&\leq
\int_0^\infty \abs{w(t)^{\eps}-w(t)^{-\eps}} \frac{\dd t}{t^{\beta}(t+1)}.
\end{align*}
Here the right hand side is independent of $s$. Using dominated convergence, we find that the right hand side tends to zero as $\eps\to0$. We obtain \eqref{eq:d5}.

Next, for $b_\eps(t)$, changing integration variable and using \eqref{eq:d8} we find 
\begin{align*}
b_\eps(t)&=\int_0^\infty \abs{(w(t)/w(st))^\eps-1} (w(t)/w(st))^{2\alpha} \frac{\dd s}{s^\beta(1+s)}
\\
&\leq
\int_0^\infty \abs{(w(t)/w(st))^\eps-1} \frac{\dd s}{w(s)^{2\alpha}s^\beta(1+s)}.
\end{align*}
Using \eqref{eq:d10} with $t$ and $s$ interchanged, we obtain 
\begin{align*}
b_\eps(t)&\leq
\int_0^\infty \abs{w(s)^{\eps}-w(s)^{-\eps}} \frac{\dd s}{w(s)^{2\alpha}s^\beta(1+s)}.
\end{align*}
Here the right hand side is independent of $t$. Using dominated convergence, we find that it tends to zero as $\eps\to0$. We obtain \eqref{eq:d6}. The proof of Lemma~\ref{lma:d1} is complete. \qed

\section{Proof of Theorem~\ref{thm:b5}(ii) (BGK theory)}
\label{sec:h}

\subsection{General part}
Here our construction is entirely parallel to the case of the Schr\"odinger operator, dubbed BGK theory in  \cite[Section~C.5]{Simon}. Below $w_{\log}$ is the operator of multiplication by the logarithmic weight
\[
w_{\log}(t)=\frac1{\sqrt{1+(\log t)^2}}, \quad t>0.
\]
Of crucial importance is the following simple fact, which we state as a lemma for emphasis. 

\begin{lemma}\label{lma:h0}
Let $H$ be a bounded positive Hankel operator. 
The operator ${w_{\log}}H{w_{\log}}$ is trace class. 
\end{lemma}
\begin{proof}
From the factorisation \eqref{eq:g3a} we find
\[
{w_{\log}}H{w_{\log}}=(L_\mu{w_{\log}})^*(L_\mu{w_{\log}}),
\]
and so it suffices to prove that $L_\mu{w_{\log}}$ is Hilbert-Schmidt. This is a direct calculation:
\begin{align*}
\int_0^\infty \int_0^\infty \ee^{-2t\lambda}({w_{\log}}(t))^{2}\dd t\, \dd\mu(\lambda)
&=
\int_0^\infty ({w_{\log}(t)})^{2}h(2t)\dd t
\leq
C_h 
\int_0^\infty \frac{({w_{\log}}(t))^{2}}{2t}\dd t
\\
&=
C_h
\int_0^\infty 
\frac{1}{2t(1+(\log t)^2)}\dd t<\infty.
\qedhere
\end{align*}
\end{proof}

We recall that a Borel measure $\rho$ on $\bbR$ is called a \emph{spectral measure} for a self-adjoint operator $H$, if $\rho$ is mutually absolutely continuous with the projection-valued spectral measure $\mathcal E_H$ of $H$. 

The following is essentially a copy of parts (i),(ii) and (iv) of \cite[Theorem~C.5.2]{Simon}, including the proof. 

\begin{lemma}\label{thm:h1}
Let $H$ be a positive bounded Hankel operator. 
There exists a spectral measure $\rho$ of $H^\perp$ and for $\rho$-a.e. $E>0$ a function $F(t_1,t_2;E)$ on $\bbR_+\times\bbR_+$ such that $F(t_1,t_2;E)$ is jointly measurable in $t_1,t_2,E$ and 
\begin{enumerate}[\rm (i)]
\item
For $\rho$-a.e. $E>0$:
\begin{equation}
\int_0^\infty \int_0^\infty \abs{F(t_1,t_2;E)}^2{w_{\log}}(t_1)^{2}{w_{\log}}(t_2)^{2}\dd t_1\, \dd t_2\leq1.
\label{eq:h3}
\end{equation}
\item
For any bounded Borel function $g$ on $\bbR$ with $g(0)=0$ and for any 
function $f$ with $f/{w_{\log}}\in L^2(\bbR_+)$, 
one has
\[
(g(H)f)(t_1)
=
\int_0^\infty g(E) \left\{ \int_0^\infty F(t_1,t_2;E)f(t_2)\dd t_2\right\}\dd\rho(E), 
\quad \text{ a.e. $t_1>0$.}
\]
\end{enumerate}
\end{lemma}
\begin{proof}
We follow the proof of \cite[Theorem~C.5.2]{Simon}.
Consider the operator-valued measure
\[
A(\Delta)={w_{\log}}\mathcal{E}_H(\Delta){w_{\log}}, \quad
\quad \Delta\subset \bbR_+
\]
on $\bbR_+$, where $\mathcal{E}_H$ is the projection-valued spectral measure of $H$ and each Borel set $\Delta$ is separated from the origin. By Lemma~\ref{lma:h0}, each  $A(\Delta)$ is trace-class. Let us define the scalar Borel measure $\rho$ on $\bbR_+$ by 
\[
\rho(\Delta)=\Tr A(\Delta).
\]
From the definition it is clear that $\rho$ is a spectral measure of $H^\perp$. 
By \cite[Theorem~C.5.1]{Simon} (the ``trace class Radon--Nikodym theorem''), 
 there exists a $\rho$-measurable trace-class positive operator $a(E)$ such that
\[
A(\Delta)=\int_{\Delta} a(E)\dd\rho(E)
\]
and
\begin{equation}
\Tr a(E)=1, \quad \text{$\rho$-a.e. $E>0$.}
\label{eq:h1}
\end{equation} 
Let  $a(t_1,t_2;E)$ be the integral kernel of $a(E)$. From \eqref{eq:h1} we find
\begin{equation}
\int_0^\infty \int_0^\infty \abs{a(t_1,t_2;E)}^2 \dd t_1\, \dd t_2
=\Tr(a(E)^2)\leq (\Tr a(E))^2=1. 
\label{eq:h2}
\end{equation}

We define
\[
F(t_1,t_2;E)=
\frac{a(t_1,t_2;E)}{{w_{\log}}(t_1){w_{\log}}(t_2)}
\]
for a.e. $t_1,t_2,E$. Then the measurability of $F$ is evident and (i) follows from \eqref{eq:h2}. 
Writing 
\[
{w_{\log}}g(H){w_{\log}}=\int_0^\infty g(E)a(E)\dd\rho(E)
\]
and applying to $f/{w_{\log}}$, we obtain (ii). 
\end{proof}

\subsection{Completing the proof}
The following lemma is trivial, but we state it explicitly in order to emphasize the concepts of weak and strong forms of the eigenvalue equation for $H$. We recall our notation $C_{\comp,0}^{\infty}(\bbR_+)$, see \eqref{eq:g2a}.
\begin{lemma}\label{lma:h1}
Let $H$ be a positive bounded Hankel operator. 
Let $f$ be a function satisfying \eqref{eq:b7a} such that for some $E\in\bbC\setminus\{0\}$ and for all $\psi\in C_{\comp,0}^{\infty}(\bbR_+)$
\[
\int_0^\infty f(t)H\psi(t)\dd t=E\int_0^\infty f(t)\psi(t)\dd t.
\]
Then $f$ satisfies the Hankel equation $Hf=Ef$. 
\end{lemma}
\begin{proof}
Interchanging the order of integration by Fubini, we find
\[
\int_0^\infty Hf(t)\psi(t)\dd t=E\int_0^\infty f(t)\psi(t)\dd t
\]
for all $\psi\in C_{\comp,0}^{\infty}(\bbR_+)$, and therefore $Ef(t)-Hf(t)=\text{const}$. Since $Hf(t)\to0$ as $t\to\infty$ (see Proposition~\ref{thm:b4}(ii)), we find that $f(t)\to\text{const}/E$ as $t\to\infty$.
By condition \eqref{eq:b7a}, from here we find $\text{const}=0$ and so $Hf=Ef$, as claimed. 
\end{proof}

\begin{remark*}
In the proof of Theorem~\ref{thm:b5}(ii) below we have $\psi\in C_{\comp}^{\infty}$, but in Section~\ref{sec:g}  we will need the case of $\psi\in C_{\comp,0}^{\infty}$. 
\end{remark*}

\begin{proof}[Proof of Theorem~\ref{thm:b5}(ii) (BGK theory)]
Let $F$ and $\rho$ be as in Lemma~\ref{thm:h1}.
For $\psi\in C_{\comp}^{\infty}(\bbR_+)$, we write $g(H)H\psi$ in two ways, using Lemma~\ref{thm:h1}(ii):
\begin{align*}
\int_0^\infty g(E)& \left\{ \int_0^\infty F(t_1,t_2;E)(H\psi)(t_2)\dd t_2\right\}\dd\rho(E)
\\
&=\int_0^\infty g(E)E \left\{ \int_0^\infty F(t_1,t_2;E)\psi(t_2)\dd t_2\right\}\dd\rho(E)
\end{align*}
for a.e. $t_1>0$. Since $g$ is arbitrary, this implies 
\begin{align*}
\int_0^\infty F(t_1,t_2;E)(H\psi)(t_2)\dd t_2
=E  \int_0^\infty F(t_1,t_2;E)\psi(t_2)\dd t_2
\end{align*}
for a.e. $t_1>0$ and $\rho$-a.e. $E>0$. By \eqref{eq:h3} and Cauchy-Schwarz the function $F$ satisfies
\[
\int_0^\infty \abs{F(t_1,t_2;E)}(1+t_2)^{-1}\dd t_2<\infty
\]
almost everywhere. 
By Lemma~\ref{lma:h1}, for a.e. $t_1>0$ and $\rho$-a.e. $E>0$, the function $f(t_2)=F(t_1,t_2;E)$ is a solution to the Hankel equation $Hf=Ef$. By varying $t_1$, we can ensure that $f$ is non-zero. 

Let us prove that $f\in {\calB}_0$. By \eqref{eq:h3}, we have
\[
\int_0^\infty \abs{f(t)}^2({w_{\log}}(t))^{2} \dd t<\infty.
\]
Let us convert this into a pointwise estimate for $f$. Using the Hankel equation and Cauchy-Schwarz, we find
\[
\abs{f(t)}\leq \frac{C_h}{E}\int_0^\infty \frac{\abs{f(s)}}{t+s}\dd s
\leq
\frac{C_h}{E}\left(\int_0^\infty \abs{f(s)}^2({w_{\log}}(s))^{2\eps} \dd s\right)^{1/2}g(t),
\]
where
\[
g(t)=\left(\int_0^\infty \frac{1+(\log s)^2}{(t+s)^{2}}\dd s\right)^{1/2}, \quad t>0.
\]
Elementary analysis shows that $g\in {\calB}_0$, and so $f\in {\calB}_0$. 
The proof is complete. 
\end{proof}

\section{Proof of Theorem~\ref{thm:b5}(iii) (eigenspace dimension $\leq2$) }
\label{sec:f}
\subsection{Preamble}
The basis of our proof is the rank two commutation relation \eqref{eq:f10} below. Let us indicate the origin of \eqref{eq:f10}, as without an explanation it looks like a ``miracle''. Consider the orthonormal basis in $L^2(\bbR_+)$ 
\[
e_n(t)=\sqrt{2}\ee^{-t}L_n(2t), \quad n\geq0,
\]
where $L_n$ are the standard Laguerre polynomials. A calculation shows (see e.g. \cite[Theorem~8.9]{Peller}) that the matrix of any integral Hankel operator in this basis is a Hankel matrix, i.e. 
\[
\jap{He_n,e_m}=h_{n+m}, \quad n,m\geq0,
\]
for some sequence $\{h_{n}\}_{n=0}^\infty$ ($h_n$ can be expressed in terms of the kernel function of $H$). 

Let us for a moment discuss the theory of Hankel matrices 
\[
\{h_{n+m}\}_{n,m=0}^\infty,
\]
viewed as operators on $\ell^2(\bbZ_+)$. Denoting by $\calS$ the standard right shift operator in $\ell^2(\bbZ_+)$, 
\begin{equation}
(\calS x)_n=x_{n-1}, \quad (\calS x)_0=0,
\label{eq:f0b}
\end{equation} 
one can easily check that  Hankel matrices satisfy the commutation relation 
\[
\calS^*\{h_{n+m}\}=\{h_{n+m}\}\calS.
\]
From this commutation relation it is easy to get \cite[formula (1.4)]{MPT}
\begin{equation}
(\calS-\calS^*)\{h_{n+m}\}=-\{h_{n+m}\}(\calS-\calS^*)+\text{(rank two operator)}.
\label{eq:f0}
\end{equation}
The relation \eqref{eq:f10} below is the mapping of \eqref{eq:f0} from the class of Hankel matrices to the class of integral Hankel operators. Rather than constructing this mapping, below we derive \eqref{eq:f10} directly, using \eqref{eq:f0} as a template.

\subsection{The operators $\calS$ and $A$}
Let $A$ be the operator in $L^2(\bbR_+)$ defined by 
\begin{equation}
(Af)(t)=2 e^{-t}\int_0^t e^sf(s)\dd s.
\label{eq:f3}
\end{equation}
Applying the Schur test, it is easy to see that $A$ is bounded. 
It is clear that the adjoint of $A$ is given by 
\[
(A^*f)(t)=2 e^{t}\int_t^\infty e^{-s}f(s)\dd s.
\]
By a slight abuse of notation, we will denote 
\begin{equation}
\calS=I-A;
\label{eq:f4a}
\end{equation}
as we shall see, this is the exact analogue of the shift operator \eqref{eq:f0b}.

\begin{lemma}
Let $\calS$ be the operator defined by \eqref{eq:f3}, \eqref{eq:f4a}. Then 
$\calS$ satisfies 
\begin{align}
\calS^*\calS&=I, \quad 
\calS\calS^*=I-\jap{\cdot,e}e, \quad\text{where}\quad e(t)=\sqrt{2}e^{-t}.
\label{eq:f6}
\end{align}
\end{lemma}
\begin{proof}
We denote 
\[
a\wedge b=\min\{a,b\}, \quad a\vee b=\max\{a,b\}.
\]
By inspection, we have 
\[
(A+A^*)f(t)=2\int_0^\infty e^{-\abs{t-s}}f(s)\dd s,
\]
and furthermore 
\begin{align*}
A^*Af(t)&=4e^t\int_t^\infty e^{-s} \left(e^{-s}\int_0^s e^u f(u)\dd u\right)\dd s
=4\int_0^\infty e^{t+u}f(u)\left(\int_{t\vee u}^\infty e^{-2s}\dd s\right)\dd u
\\
&=2\int_0^\infty e^{t+u-2t\vee u}f(u)\dd u
=2\int_0^\infty e^{-\abs{t-u}}f(u)\dd u.
\end{align*}
Similarly, 
\begin{align*}
AA^*f(t)&=4e^{-t}\int_0^t e^s \left(e^{s}\int_s^\infty e^{-u} f(u)\dd u\right)\dd s
\\
&=4\int_0^\infty e^{-(t+u)}f(u)\left(\int_0^{t\wedge u} e^{2s}\dd s\right)\dd u
=2\int_0^\infty e^{-(t+u)}f(u)\bigl(e^{2t\wedge u}-1)\dd u
\\
&=2\int_0^\infty e^{-\abs{t-u}}f(u)\dd u-2\int_0^\infty e^{-t}e^{-u}f(u)\dd u.
\end{align*}
From here we conclude that 
\begin{align*}
A^*A&=A+A^*,
\\
AA^*&=A+A^*-\jap{\cdot,e}e,
\end{align*}
which is equivalent to \eqref{eq:f6}.
\end{proof}

\subsection{The rank two commutator identity}
Let $\mu$ be a Carleson measure and let $h=h_\mu$, $H=H_\mu$. Let us denote
\[
p=\calS He,
\]
where $e$ is the function defined in \eqref{eq:f6} and $\calS$ is defined by \eqref{eq:f3}, \eqref{eq:f4a}.
\begin{theorem}\label{thm:f2}
We have the commutator identities
\begin{align}
\calS^*H&=H\calS,
\label{eq:f7}
\\
\calS H&=H\calS^*-\jap{\cdot,p}e+\jap{\cdot,e}p,
\label{eq:f8}
\end{align}
and
\begin{equation}
\boxed{(\calS-\calS^*)H=-H(\calS-\calS^*)-\jap{\cdot,p}e+\jap{\cdot,e}p.}
\label{eq:f10}
\end{equation}
\end{theorem}
Formula \eqref{eq:f10} is our main tool in this section.
\begin{proof}
The proof of \eqref{eq:f7} is a direct calculation. Indeed, we have 
\begin{align*}
(HAf)(t)&=2\int_0^\infty h(t+s)\left(e^{-s}\int_0^s e^u f(u)\dd u\right)\dd s
\\
&=2\int_0^\infty f(u)\left(e^u\int_u^\infty e^{-s} h(t+s)\dd s\right)\dd u
\\
&=2\int_0^\infty f(u)\left(\int_0^\infty e^{-s}h(t+s+u)\dd s\right)\dd u
\end{align*}
and similarly
\begin{align*}
(A^*Hf)(t)&=2e^t \int_t^\infty e^{-s}\left(\int_0^\infty h(s+u)f(u)\dd u\right)\dd s
\\
&=2\int_0^\infty f(u)\left(e^t\int_t^\infty e^{-s}h(s+u)\dd s\right)\dd u
\\
&=2\int_0^\infty f(u)\left(\int_0^\infty e^{-s}h(t+s+u)\dd s\right)\dd u.
\end{align*}
Here the right hand sides coincide, and so we get \eqref{eq:f7}. 

Using the identity $I = \calS\calS^* + \jap{\cdot, e}e$, we find:
\begin{align*}
H\calS^* &= (\calS\calS^* + \jap{\cdot, e}e) H\calS^* 
&= \calS(\calS^*H)\calS^* + \jap{H\calS^* \cdot, e}e 
&= \calS(H\calS)\calS^* + \jap{\cdot, \calS He}e.
\end{align*}
Using $p = \calS He$ and substituting $\calS\calS^* = I - \jap{\cdot, e}e$ into the first term above, we find:
\begin{align*}
H\calS^* &= \calS H(I - \jap{\cdot, e}e) + \jap{\cdot, p}e 
= \calS H - \jap{\cdot, e}\calS He + \jap{\cdot, p}e \\
&= \calS H - \jap{\cdot, e}p + \jap{\cdot, p}e.
\end{align*}
Rearranging yields \eqref{eq:f8}. 
The last relation \eqref{eq:f10} is obtained by subtracting \eqref{eq:f7} from \eqref{eq:f8}. 
\end{proof}

\subsection{Extending to ${\calB}_0$}
Our aim is to extend the rank-two commutator identity \eqref{eq:f10} to the space ${\calB}_0$. 
As an intermediate step, we consider the spaces ${\calB}_\eps$ with $\eps\in (0,\frac12)$, with the norm 
\[
\norm{f}_{{\calB}_\eps}=
\sup_{0<t<1}t^{\frac12+\eps}\abs{f(t)}+\sup_{t>1}t^{\frac12-\eps}\abs{f(t)}.
\]

Below $H$, $h$ and $p$ are as in Theorem~\ref{thm:f2}. 
\begin{lemma}\label{lma:f3}
For any $\eps\in (0,\frac12)$, the operators $H$, $\calS$, $\calS^*$ are bounded on ${\calB}_\eps$. 
\end{lemma}
\begin{proof}
For $f\in {\calB}_\eps$, we have 
\[
\abs{Hf(t)}
\leq 
C_{h}\norm{f}_{{\calB}_\eps}
\left\{
\int_0^1 (s+t)^{-1}s^{-\frac12-\eps}\dd s
+
\int_1^\infty (s+t)^{-1}s^{-\frac12+\eps}\dd s
\right\}
\]
A calculation shows that the sum in $\{\cdots\}$ is in ${\calB}_\eps$. 
Thus, $H$ is bounded on ${\calB}_\eps$. 
Similarly, 
\begin{align*}
\abs{Af(t)}+\abs{A^*f(t)}
&\leq
2\int_0^\infty \ee^{-\abs{s-t}}\abs{f(s)}\dd s
\\
&\leq
2\norm{f}_{{\calB}_\eps}
\left\{\int_0^1  \ee^{-\abs{s-t}}s^{-\frac12-\eps}\dd s
+
\int_1^\infty  \ee^{-\abs{s-t}}s^{-\frac12+\eps}\dd s\right\}
\end{align*}
and again by inspection the term in $\{\cdots\}$ is in ${\calB}_\eps$. 
Thus, $A$ and $A^*$ (and therefore $\calS$, $\calS^*$) are bounded on ${\calB}_\eps$. 
\end{proof}

\begin{lemma}\label{lma.f4}
The function $p=\calS He$ satisfies
\begin{equation}
\abs{p(t)}\leq C\frac{\abs{\log t}}{1+t}, \quad t>0.
\label{eq:f11}
\end{equation}
\end{lemma}
\begin{proof}
Consider the function $He$; we have
\[
He(t)=\sqrt{2}\int_0^\infty h(t+s)\ee^{-s}\dd s\leq C_h\int_0^\infty \frac{\ee^{-s}}{t+s}\dd s,
\]
and therefore 
\[
He(t)=O(\abs{\log t}), \quad t\to0_+,
\qquad
He(t)=O(t^{-1}), \quad t\to\infty.
\]
Applying $\calS=I-A$, after a short calculation we obtain \eqref{eq:f11}.
\end{proof}

From the explicit form $e(t)=\sqrt{2}\ee^{-t}$ and from \eqref{eq:f11}, we see that the rank-one operators
\[
f\mapsto \jap{f,e}p
\quad\text{ and }\quad
f\mapsto \jap{f,p}e
\]
extend to  ${\calB}_\eps$ for any $\eps\in(0,\frac12)$.

\begin{lemma}
Let $H$ be a bounded positive Hankel operator and let $e(t)=\sqrt{2}e^{-t}$ and $p=\calS He$. Then the rank-two commutator identity \eqref{eq:f10} holds true on the space ${\calB}_\eps$ for any $\eps\in(0,\frac12)$.
\end{lemma}
\begin{proof}
Let $f\in {\calB}_\eps$; denote
\begin{align*}
F(t)=&
((\calS^*-\calS)H)f(t)+(H(\calS^*-\calS)f)(t)
\\
&-e(t)\int_0^\infty f(s)p(s)\dd s+p(t)\int_0^\infty f(s)e(s)\dd s, \quad t>0.
\end{align*}
By the previous two lemmas, $F\in {\calB}_\eps$; we need to prove that $F=0$. 
For any $\psi\in C_{\comp}^\infty(\bbR_+)$, applying Fubini we  find that
\begin{equation}
\int_0^\infty F(t)\overline{\psi(t)}\dd t=
-\int_0^\infty f(t)\overline{\Psi(t)}\dd t,
\label{eq:f11a}
\end{equation}
where 
\[
\Psi=(\calS^*-\calS)H\psi+H(\calS^*-\calS)\psi-\jap{\psi,p}e+\jap{\psi,e}p.
\]
By Theorem~\ref{thm:f2}, we find $\Psi=0$. Thus, the left-hand side of \eqref{eq:f11a} vanishes for any $\psi\in C_{\comp}^\infty(\bbR_+)$, which forces $F=0$. 
\end{proof}

\subsection{Proof of Theorem~\ref{thm:b5}(iii) (eigenspace dimension $\leq2$)}
Assume, to get a contradiction, that for some $E>0$ the linear space of solutions to $Hf=Ef$ in ${\calB}_0$ has dimension $>2$. Then there exists a solution $f\in{\calB}_0$ such that $\jap{f,e}=\jap{f,p}=0$. From the rank-two commutator identity \eqref{eq:f10} we find 
\[
Hu=-Eu, \quad u=(\calS^*-\calS)f. 
\]
By Lemma~\ref{lma:f3} we find that $u\in {\calB}_0$. Thus, $u$ is an eigenfunction corresponding to the eigenvalue $-E$. Since $H$ is positive, $-E$ is not in the spectrum of $H$. Applying part (i) of the theorem, we find $u=0$. 

It remains to check that the kernel of $\calS^*-\calS$ in ${\calB}_0$ is trivial, i.e. $Af=A^*f$ implies $f=0$. By differentiating, we find
\begin{align*}
(Af)'(t)&=2f(t)-(Af)(t), 
\\
(A^*f)'(t)&=-2f(t)+(A^*f)(t).
\end{align*}
Combining this with $Af=A^*f$, we find 
\[
4f=(A+A^*)f.
\]
Differentiating again gives
\[
4f'=(2f-Af)+(-2f+A^*f)=(-A+A^*)f=0,
\]
and so $f=\text{const}$. From here it easily follows that $f=0$. The proof of Theorem~\ref{thm:b5}(iii) is complete. \qed

\section{Finite semi-regular case}
\label{sec:g}

Throughout this section we assume that $\mu$ is  finite and satisfies the Carleson condition \eqref{eq:Carleson}. In this case we have
\[
h(t)\leq \mu(\bbR_+)\quad\text{ and }\quad h(t)\leq C_\mu/t, \quad t>0
\]
and in particular, $h\in L^2(\bbR_+)$. 

\subsection{Some heuristics}\label{sec:g0}
Assume that $\varphi_E$ solves $H\varphi_E=E\varphi_E$ for some $E>0$. Let us differentiate the Hankel equation,
\[
\int_0^\infty h'(t+s)\varphi_E(s)\dd s=E\varphi_E'(t),
\]
and then integrate by parts:
\begin{equation}
-\int_0^\infty h(t+s)\varphi_E'(s)\dd s-\varphi_E(0)h(t)=E\varphi_E'(t). 
\label{eq:b13b}
\end{equation}
This can be written as
\begin{equation}
(H+E)\varphi_E'=-\varphi_E(0)h.
\label{eq:b13a}
\end{equation}
Observe that since $E>0$, now the operator $(H+E)$ on the left-hand side is invertible.

Assuming that our integration by parts calculation was legitimate and that $\varphi_E(0)$ is well-defined, from here we learn two things. Firstly, the space of solutions to $H\varphi_E=E\varphi_E$ is one-dimensional (or zero). Indeed, otherwise we could choose a non-zero solution with $\varphi_E(0)=0$ which would imply $\varphi_E'=0$, hence $\varphi_E=0$, which is a contradiction. 

Secondly, the unique (up to scaling) solution can be found explicitly from  \eqref{eq:b13a}. Indeed, let us, for example, normalise $\varphi_E$ by setting $\varphi_E(0)=-1$. Then we find $\varphi_E$ by 
\begin{equation}
\boxed{
\varphi_E(t)=-1+\int_0^t \varphi_{E}'(s)\dd s, \quad\text{ where }\quad \varphi_{E}'=(H+E)^{-1}h.}
\label{eq:b13}
\end{equation}

In this section we reverse the logic of this calculation. Namely, for all $E\in\bbC\setminus(-\spec(H))$ we \emph{define} the function $\varphi_E$ by \eqref{eq:b13}. Of course, thus defined, $\varphi_{E}$ is not necessarily in $L^2(\bbR_+)$. However, we will see that
\begin{itemize}
\item
for each $E>0$, the function $\varphi_E$ is a \emph{weak solution} to the Hankel equation;
\item
for $H$-spectrally almost every $E>0$, we have $\varphi_E\in {\calB}_0$ and $H\varphi_E=E\varphi_E$;
\item
$\varphi_E$ effects the diagonalisation of the operator $H$ in a way similar to \eqref{eq:a7}.
\end{itemize}

\subsection{The solution $\varphi_E$}
The following result should be considered as a preliminary (and its proof is relatively straightforward). It justifies the usage of $\varphi_E$ as the ``generalised eigenfunction'' of $H$. 

\begin{lemma}\label{prp:b10}
Let $H$ be a bounded positive Hankel operator such that the measure $\mu$ is finite, and let $\varphi_E$ be defined by \eqref{eq:b13}. Then:
\begin{enumerate}[\rm (i)]
\item
For any $E\in\bbC\setminus(-\spec(H))$, we have 
\begin{equation}
\varphi_{E}(t)=O(t^{1/2}), \quad t\to\infty.
\label{eq:b10}
\end{equation}
\item
For any $E\in\bbC\setminus(-\spec(H))$, 
the function $\varphi_E$ is a weak solution to the eigenvalue equation $H\varphi_E=E\varphi_E$ in the following sense:
\begin{equation}
\int_0^\infty \varphi_E(t)(Hf)(t)\dd t=E\int_0^\infty \varphi_E(t)f(t)\dd t, \quad \forall f\in C_{\comp,0}^\infty(\bbR_+).
\label{eq:b12}
\end{equation}
\item
For $E>0$, if $\varphi_E\in {\calB}_0$, then $H\varphi_E=E\varphi_E$ and $E\in\spec(H)$.
\item
For $E>0$, if $Hf=Ef$ with $f\in {\calB}_0$, then the limit $f(0)$ is well-defined and non-zero. Moreover, $f$ is collinear with $\varphi_E$:
\[
f=-f(0)\varphi_E.
\]
\item
For $H$-spectrally almost every $E>0$, we have $\varphi_E\in {\calB}_0$ and $H\varphi_E=E\varphi_E$. 
\end{enumerate}
\end{lemma}

\begin{remark*}
We note that \eqref{eq:b10} is insufficient for the convergence of the integral in the Hankel equation \eqref{eq:b7}, and therefore, for a general $E$, we cannot assert that $\varphi_E$ satisfies $H\varphi_E=E\varphi_E$.  On the other hand, integrating by parts, one checks that 
\[
(Hf)(t)=O(t^{-2}), \quad t\to\infty
\]
for $f\in C^\infty_{\comp,0}(\bbR_+)$, and therefore the integral in the left-hand side of \eqref{eq:b12} is well-defined. 
\end{remark*}

\subsection{The eigenfunction expansion of $H$}

For any function $f\in C^\infty_{\comp}(\bbR_+)$ and for $E>0$ we set 
\begin{equation}
\boxed{
(\Phi f)(E)=-\frac1{\sqrt{E}}\int_0^\infty \varphi_{E}(t) f(t)\dd t.}
\label{eq:g7}
\end{equation}
Here $-1/\sqrt{E}$ is a normalisation factor, which will be explained in Remark~\ref{rmk:g4} below. The main result of this section is:

\begin{theorem}\label{thm:b7}
Let $H$ be a bounded positive Hankel operator such that the measure $\mu$ in the Laplace transform representation \eqref{eq:a6} is finite. Let $\Phi$ be the integral transform defined by \eqref{eq:g7} on $C^\infty_{\comp}(\bbR_+)$. Then there exists a finite measure $\sigma$ on $(0,\infty)$ such that $\Phi$ extends to a partial isometry from $L^2(\bbR_+)$ onto $L^2(\sigma)$, with $\Ker\Phi=\Ker H$, i.e. the Parseval identity 
\begin{equation}
\int_0^\infty \abs{(\Phi f)(E)}^2\dd\sigma(E)=\norm{P_{\overline{\Ran H}}f}^2, 
\quad
f\in C^\infty_{\comp}(\bbR_+),
\label{eq:g7b}
\end{equation}
holds true, where $P_{\overline{\Ran H}}$ is the orthogonal projection onto $\overline{\Ran H}$ in $L^2(\bbR_+)$. Thus extended, $\Phi$ intertwines $H$ and the multiplication by the independent variable in $L^2(\sigma)$, i.e. 
\[
(\Phi H f)(E)=E(\Phi f)(E) \quad \text{for $\sigma$-a.e. $E>0$.}
\]
\end{theorem}

The theorem says that  $H^\perp$ is unitarily equivalent to the multiplication by the independent variable in $L^2(\sigma)$, i.e. $\sigma$ is a \emph{spectral measure} of $H^\perp$. Moreover, $\Phi$ gives a concrete representation for the eigenfunction expansion of $H$. 

\subsection{The cyclic element $v_\mu$}
Here we explain the meaning of the measure $\sigma$ in Theorem~\ref{thm:b7}. 
The existence of $\sigma$ will be provided by the results of \cite{PuTreil}, which we recall here. Of importance for us is that the element $v_\mu\in L^2(\bbR_+)$, \emph{formally} defined by 
\[
v_\mu=H^{1/2}\delta
\]
(where $\delta$ is the delta-function at the origin) is cyclic for $H^\perp$. More precisely, the following statement is proved in \cite{PuTreil}.

\begin{proposition}\cite{PuTreil}\label{thm:g-PT1}
Let $H$ be a positive bounded Hankel operator  such that the measure $\mu$ in the representation \eqref{eq:a6} is finite. 
Then there exists a unique element $v_\mu\in\overline{\Ran H}$ such that for any $g\in C^\infty_{\comp}(\bbR_+)$ we have 
\[
\jap{v_\mu,H^{1/2}g}=\int_{0}^{\infty}h(t)\overline{g(t)}\dd t.
\]
The element $v_\mu$ is cyclic for $H^\perp$ and $\norm{v_\mu}^2=\mu(\bbR_+)$. 
\end{proposition}
We will see that the measure $\sigma$ of Theorem~\ref{thm:b7} is the spectral measure of $H^\perp$ corresponding to the element $v_\mu$, i.e. 
\[
\sigma(\Delta)=\jap{{\mathcal E}_{H_\mu^\perp}(\Delta)v_\mu,v_\mu}, 
\quad
\Delta\subset \bbR_+.
\]

\begin{remark}\label{rmk:g4}
The normalisation factor $-1/\sqrt{E}$ in the definition \eqref{eq:g7} of $\Phi$ ensures that the target space of $\Phi$ is $L^2(\sigma)$. Without this factor, the target space would have been $L^2(\widetilde{\sigma})$, with 
\[
\dd\widetilde{\sigma}(E)=\frac{\dd\sigma(E)}{E}, \quad E>0.
\]
While $\sigma$ is the spectral measure of $H$ corresponding to the element $v_\mu$, in the same way the measure $\widetilde{\sigma}$ is the spectral measure of $H$ corresponding to the ``element'' $\delta$, i.e. the delta-function at the origin. 

The minus sign in \eqref{eq:g7} is of no importance; we choose the minus sign in order to satisfy the relation 
\[
\Phi v_\mu=1.
\]
\end{remark}

\subsection{Example: Mehler's operator}
The Hankel operator with the kernel function 
\[
h(t)=1/(2+t)=\int_0^\infty \ee^{-t\lambda}\dd\mu(\lambda), \quad
\dd\mu(\lambda)=\ee^{-2\lambda}\dd\lambda
\]
is known as  \emph{Mehler's operator}. In fact, Mehler \cite{Mehler} discovered the formula \cite[Eq. (14.20.13)]{DLMF}
\begin{equation}
\int_0^\infty 
\frac{P_{-\frac12+\ii k}(1+s)}{2+t+s}\dd s=\pi\sech(\pi k) P_{-\frac12+\ii k}(1+t), \quad 
t>0, \quad k>0,
\label{eq:g21}
\end{equation}
where $P_\nu$ is the Legendre function; in the particular case $\nu=\tfrac12+\ii k$, Legendre functions are known as conical functions.

We pause to recall the necessary properties of the Legendre function, see \cite[Chapter 3]{BE} and \cite[Section 14]{DLMF}. 
The Legendre function $P_\nu$ is an entire function of $\nu$ and can be represented as
\[
P_\nu(x)=\frac1\pi \int_0^\pi \left(x+\sqrt{x^2-1}\cos y\right)^\nu\dd y, \quad x>1.
\]
From here it can be seen that $P_\nu(x)$ is holomorphic in $x$ near $x=1$ (i.e. near $t=0$ when $x=1+t$) and satisfies 
\begin{equation}
P_\nu(1)=1
\label{eq:g29}
\end{equation}
for all $\nu$. 
Moreover, one has \cite[Formula (14.8.2)]{DLMF}
\begin{equation}
P_{\nu}(x)\sim \frac{\Gamma(\nu+\tfrac12)}{\sqrt{\pi}\Gamma(\nu+1)}(2 x)^{\nu}, \quad x\to\infty, \quad \Re \nu>-\tfrac12.
\label{eq:g30}
\end{equation}
The asymptotics of the conical functions $P_{\nu}(x)$, $\nu=-\tfrac12+\ii k$, as $x\to\infty$  involves a linear combination of two terms with $x^{-\frac12\pm\ii k}$, but for us the estimate 
\begin{equation}
P_{-\frac12+\ii k}(x)=O(x^{-\frac12}), \quad x\to\infty,
\label{eq:g31}
\end{equation}
will suffice.

We come back to Mehler's operator. Formula \eqref{eq:g21} diagonalises Mehler's operator and shows that its spectrum is purely a.c., has multiplicity one and coincides with the interval $[0,\pi]$, see e.g. \cite{Ya1}. Note that $P_{-\frac12-\ii k}=P_{-\frac12+\ii k}$, i.e. changing the sign of $k$ does not give a new solution to \eqref{eq:g21} -- compare with \eqref{eq:a6a}.

From \eqref{eq:g31} it follows that the conical functions $P_{-\frac12+\ii k}(1+t)$ are in $\calB_0$. 
From here and Lemma~\ref{prp:b10}(iv), using \eqref{eq:g29}, we find 
\begin{equation}
\varphi_E(t)=-P_{-\frac12+ik}(1+t),\quad E=\pi\sech(\pi k), \quad k>0. 
\label{eq:g23}
\end{equation}
Using the analytic continuation in $\nu=-\frac12+\ii k$, one can extend \eqref{eq:g23} to other values of $E$. We confine ourselves to two cases:

Case 1: $k=-\ii\kappa$, $\kappa\in(0,\tfrac12)$, i.e. $\nu=-\tfrac12+\kappa\in (-\tfrac12,0)$. This corresponds to \emph{positive} values $E=\pi\sec(\pi\kappa)>\pi$ outside the spectrum. Formula \eqref{eq:g30} shows that in this case $\varphi_E$ is not in $\calB_0$. 

Case 2: $k=-\ii\kappa$, $\kappa\in(\tfrac12,1)$, i.e. $\nu=-\tfrac12+\kappa\in (0,\tfrac12)$. This corresponds to \emph{negative} values $E=\pi\sec(\pi\kappa)<-\pi$. 
Using \eqref{eq:g30} again, we see that in this case the function $\varphi_E(t)$ is even ``worse'', as it grows at infinity:
\begin{equation}
\varphi_E(t)\to\infty, \quad t\to\infty, \quad E<-\pi.
\label{eq:g28}
\end{equation}

Let us discuss Parseval's formula \eqref{eq:g7b}. The unitarity of the Mehler-Fock transform (see e.g. \cite[Section 3.4]{Yakubovich}) can be written as
\[
\norm{f}^2
=
\int_0^\infty k\tanh(\pi k)\Abs{\int_0^\infty P_{-\frac12+\ii k}(1+t)f(t)\dd t}^2\dd k.
\]
Using $\varphi_E$ on the spectrum, this can be rewritten as
\[
\norm{f}^2
=
\int_0^\infty k\tanh(\pi k)E(k)\abs{\widetilde{f}(E(k))}^2\dd k
=
\frac1{\pi^2}
\int_0^\pi\abs{\widetilde{f}(E)}^2 \sech^{-1}(E/\pi)\dd E,
\]
where 
\[
\widetilde{f}(E)=-\frac1{\sqrt{E}}\int_0^\infty P_{-\frac12+\ii k}(1+t)f(t)\dd t.
\]
Thus, the corresponding measure $\sigma$ in this case is 
\[
\dd\sigma(E)=\frac1{\pi^2}\sech^{-1}(E/\pi)\dd E,
\]
supported on $[0,\pi]$. 

In the rest of this section we prove Lemma~\ref{prp:b10} and Theorem~\ref{thm:b7}. 

\subsection{Proof of Lemma~\ref{prp:b10}}
Part (i): by the definition of $\varphi_{E}'$, it satisfies \eqref{eq:b13b}. Since $\varphi_{E}'\in L^2(\bbR_+)$, from here by Cauchy-Schwarz we find 
\[
\varphi_{E}'(t)=O(t^{-1/2}), \quad t\to\infty,
\]
and so \eqref{eq:b10} follows. 

Part (ii):
write $f=g'$ with $g\in C_{\comp}^\infty(\bbR_+)$. Integrating by parts, we find 
$Hf=-(Hg)'$. Using this, we obtain 
\begin{align*}
\int_0^\infty {\varphi_{E}(t)}(Hg')(t)\dd t
&=
-\int_0^\infty {\varphi_{E}(t)}(Hg)'(t)\dd t
\\
&={\varphi_{E}(0)}(Hg)(0)
+\int_{0}^{\infty}{\varphi_{E}'(t)}(Hg)(t)\dd t
\\
&=-\jap{g,h}+\int_{0}^{\infty}{(H\varphi_{E}')(t)}g(t)\dd t
\\
&=
-\jap{g,h}-\int_{0}^\infty(E{\varphi_{E}'(t)}-h(t))g(t)\dd t
\\
&=
-\jap{g,h}-E\int_{0}^\infty {\varphi_{E}'(t)}g(t)\dd t +\jap{g,h}
\\
&=E\int_{0}^\infty {\varphi_{E}(t)}g'(t)\dd t,
\end{align*}
where we have used the definition of $\varphi_E$ and the normalisation $\varphi_E(0)=-1$. 

Part (iii): 
by Lemma~\ref{lma:h1}, we find that $\varphi_E$ satisfies $H\varphi_E=E\varphi_E$.  
Thus, by Theorem~\ref{thm:b5}(i), we have $E\in\spec(H)$.

Part (iv): 
Let $Hf=Ef$ with a non-zero $f\in {\calB}_0$. We need to justify the reasoning of Section~\ref{sec:g0}. 

Let us discuss the properties of $f(t)$. 
Recall that the kernel function $h$ satisfies \eqref{eq:b1}. From the Hankel equation we see that the limit $f(0)=f(0_+)$ exists.
Differentiating the Hankel equation,
\begin{equation}
\int_0^\infty h'(t+s)f(s)\dd s=E f'(t), \quad t>0,
\label{eq:e3}
\end{equation}
by \eqref{eq:b7b} we find that $f'$ is bounded and
\begin{equation}
\abs{f'(t)}\leq C
\left( 
\int_0^1 (s+t)^{-2}\dd s
+ \int_1^\infty (s+t)^{-2}s^{-\frac12+\eps}\dd s
\right)
=O(t^{-\frac{3}{2}+\eps}), \quad t\to\infty,
\label{eq:e3a}
\end{equation}
with any $\eps\in (0,\frac12)$. 
Similarly, by \eqref{eq:b1} we find 
\begin{equation}
\abs{f'(t)}\leq C
\left( 
\int_0^1 (s+t)^{-1}\dd s
+ \int_1^\infty (s+t)^{-1}s^{-\frac12+\eps}\dd s
\right)
=O(\abs{\log t}), \quad t\to0_+.
\label{eq:e3b}
\end{equation}
From \eqref{eq:e3a} and \eqref{eq:e3b} we find that $f'\in L^2(\bbR_+)$. 
Now we can integrate by parts in \eqref{eq:e3}, which yields
\begin{equation}
(H+E)f'=-f(0)h.
\label{eq:e4a}
\end{equation}
If $f(0)=0$, from here we find $f'=0$, hence $f=0$; this is a contradiction. 
If $f(0)\not=0$, we can renormalise so that $f(0)=-1$. Then  \eqref{eq:e4a} gives $f'=(H+E)^{-1}h=\varphi_E'$, and so $f=\varphi_E+C$. But also $f(0)=\varphi_E(0)=-1$, so $f=\varphi_E$, as claimed.

Part (v): by Theorem~\ref{thm:b5}(ii), for $H$-spectrally almost every $E>0$, there exists $f\in {\calB}_0$ with $Hf=Ef$. By the already proven part (iii), $f$ is then collinear with $\varphi_E$. 
\qed

\subsection{Operators in $L^2(\mu)$}
In the rest of this section we use the results and the technique of \cite{PuTreil}, which utilises analysis in the Hilbert space $L^2(\mu)=L^2(\bbR_+,\dd\mu)$. We denote the inner product of elements $f$ and $g$ in $L^2(\mu)$ by $\jap{f,g}_\mu$. We denote by $\1_\mu$ is the function on $\bbR_+$ identically equal to $1$. Since $\mu$ is finite, we have $\1_\mu\in L^2(\mu)$. Below $X_\mu$ is the operator of multiplication by the independent variable in $L^2(\mu)$:
\[
X_\mu f(x)=xf(x), 
\]
and $G_\mu$ is the integral operator in $L^2(\mu)$ given by 
\begin{equation}
\boxed{
(G_\mu f)(x)=\int_0^\infty \frac{f(y)}{x+y}\dd\mu(y). }
\label{eq:g0b}
\end{equation}
We have the important factorisation
\begin{equation}
\boxed{ G_\mu=L_\mu L_\mu^*,}
\label{eq:g3b}
\end{equation}
where $L_\mu: L^2(\bbR_+)\to L^2(\mu)$ is the Laplace transform \eqref{eq:g2}, 
see \cite[Lemma~5.1]{PuTreil}. In particular, $G_\mu$ is bounded if and only if $\mu$ satisfies the Carleson condition \eqref{eq:Carleson}.

We recall (see \eqref{eq:g3a}) that 
\[
H_\mu=L_\mu^*L_\mu.
\]
Let us compare this to \eqref{eq:g3b} and recall the well-known general operator theoretic statement:
\begin{proposition}
Let $L$ be a bounded operator in a Hilbert space. Then the operators $L^*L|_{(\Ker L)^\perp}$ and $LL^*|_{(\Ker L^*)^\perp}$ are unitarily equivalent.
\end{proposition}

It is not difficult to see that $\Ker L_\mu^*=\{0\}$, see e.g. \cite[Theorem~3.1(iv)]{PuTreil0}.
Thus, we see that the operators $G_\mu$ and $H_\mu^\perp$ are unitarily equivalent. This explains our interest in $G_\mu$.

What operator effects the unitary equivalence between $G_\mu$ and $H_\mu^\perp$? To answer this question, let us write the polar decomposition of $L_\mu^*$ as 
\begin{equation}
L_\mu^*=T_\mu G_\mu^{1/2},
\quad
L_\mu=G_\mu^{1/2}T_\mu^*,
\label{eq:g4}
\end{equation}
where $T_\mu$ is an isometry from $L^2(\mu)$ to $L^2(\bbR_+)$ with $\Ran T_\mu=\overline{\Ran H}$. From here we find that
\begin{equation}
H_\mu=T_\mu G_\mu T_\mu^*.
\label{eq:g4aa}
\end{equation}
We will come back to this relation in the next subsection. But first we need to establish some properties of  the operator $G_\mu$.

\begin{proposition}\cite[Theorem~5.8]{PuTreil}\label{thm:PuTr-2}
Let $\mu$ be a finite Carleson measure on $\bbR_+$. Then the element $\1_\mu$ is cyclic for the operator $G_\mu$. Denote by $\sigma$ the spectral measure for $G_\mu$ corresponding to the element $\1_\mu$: 
\[
\sigma(\Delta)=\jap{{\mathcal E}_{G_\mu}(\Delta)\1_\mu,\1_\mu}_{\mu},
\quad
\Delta\subset\bbR.
\]
Then there exists a unitary operator $U: L^2(\mu)\to L^2(\sigma)$ such that 
\begin{equation}
UG_\mu=X_\sigma U, \quad 
UX_\mu=G_\sigma U, \quad
U\1_\mu=\1_\sigma.
\label{eq:g7a}
\end{equation}
\end{proposition}

\begin{remark*}
When $\mu$ is a finite Carleson measure, both operators $G_\mu$ and $X_\sigma$ are bounded. The subtle point of the above proposition is that the operators $X_\mu$ and $G_\sigma$ are not necessarily bounded. 
Indeed, $\mu$ may have unbounded support and $\sigma$ does not have to be Carleson. Operators $G_\sigma$ with general (including non-Carleson) measures are defined by  \eqref{eq:g0b} initially on the dense set of bounded compactly supported functions $f$ on $\bbR_+$, and then a closure is taken. 
For the details, see \cite[Section~5]{PuTreil}. 
\end{remark*}

The following crucial lemma is new; it is not contained in the results  \cite{PuTreil}. 
\begin{lemma}\label{lma:g7}
Assume the hypothesis and notation of Proposition~\ref{thm:PuTr-2}. 
For any $f\in\Dom X_\mu$, we have
\begin{equation}
(UX_\mu f)(E)=\jap{f,(G_\mu+E)^{-1}\1_\mu}_\mu, 
\quad
\text{ for $\sigma$-a.e. $E>0$.}
\label{eq:g4a}
\end{equation}
\end{lemma}
\begin{proof}
Using \eqref{eq:g7a} and the definition of $G_\sigma$,  we find
\begin{align*}
(UX_\mu f)(E)
&=(G_\sigma Uf)(E)
=\jap{(X_\sigma+E)^{-1}Uf,\1_\sigma}_{\sigma}
=\jap{U(G_\mu+E)^{-1}f,\1_\sigma}_{\sigma}
\\
&=\jap{(G_\mu+E)^{-1}f,U^*\1_\sigma}_\mu
=\jap{(G_\mu+E)^{-1}f,\1_\mu}_\mu
\\
&=\jap{f,(G_\mu+E)^{-1}\1_\mu}_\mu,
\end{align*}
as claimed. 
\end{proof}
Observe that while the left-hand side of \eqref{eq:g4a} is defined for $\sigma$-a.e. $E>0$, the right-hand side is defined (and analytic) for all $E\in\bbC\setminus(-\spec(G_\mu))$. 

\subsection{The unitary map $\Phi_*$}
We would like to set up a  partial isometry 
\[
\Phi_*:L^2(\bbR_+)\to L^2(\sigma). 
\]
(At the end of the proof, we will see that $\Phi_*=\Phi$, hence the notation.)

Recall that $T_\mu$ is a partial isometry of \eqref{eq:g4} and \eqref{eq:g4aa}, and  $U$ is the operator of Proposition~\ref{thm:PuTr-2}.
We define 
\begin{equation}
\Phi_*:=UT_\mu^*.
\label{eq:g5}
\end{equation}
By definition, $\Phi_*$ is a partial isometry from $L^2(\bbR_+)$ onto $L^2(\sigma)$, with $\Ker\Phi_*=\Ker H$ and $\Ran \Phi_*=L^2(\sigma)$. 
Below $v_{\mu}$ is the cyclic element of Proposition~\ref{thm:g-PT1}.

\begin{lemma}\cite{PuTreil}\label{prp:g3}
The partial isometry $\Phi_*$ satisfies the properties:
\begin{equation}
\Phi_*H=X_\sigma\Phi_*, 
\quad
\Phi_*v_{\mu}=\1_\sigma.
\label{eq:g6}
\end{equation}
\end{lemma}
\begin{proof}
The proof is implicit in \cite{PuTreil}; here we explain the details. 
Combining $T_\mu^*H_\mu=G_\mu T_\mu^*$ (see  \eqref{eq:g4aa}) and $UG_\mu=X_\sigma U$ (see \eqref{eq:g7a}), we immediately obtain the first identity $\Phi_*H=X_\sigma\Phi_*$. 
Next, the identity 
\[
T_\mu\1_{\mu}=v_{\mu}
\]
is proved in  \cite[Lemma~6.1]{PuTreil}. From this identity and $U\1_\mu=\1_\sigma$ (see \eqref{eq:g7a}) we obtain the second property  $\Phi_*v_{\mu}=\1_\sigma$. 
\end{proof}

\begin{lemma}\label{prp:g3a}
For any $g\in C_{\comp}^\infty(\bbR_+)$, we have 
\begin{equation}
(\Phi_*g')(E)
=
\frac1{\sqrt{E}}\jap{g,(H+E)^{-1}h}, 
\quad
\text{ for $\sigma$-a.e. $E>0$.}
\label{eq:g6b}
\end{equation}
\end{lemma}
\begin{proof}

Using \eqref{eq:g5}, \eqref{eq:g7a} and  \eqref{eq:g4}, we find
\begin{equation}
X_{\sigma}^{1/2}\Phi_*=X_{\sigma}^{1/2}UT_{\mu}^*=UG_{\mu}^{1/2}T_{\mu}^*=UL_\mu.
\label{eq:g6a}
\end{equation}
Let us compute both sides of the last identity on the element $g'$. 
The left-hand side is
\[
(X_{\sigma}^{1/2}\Phi_*g')(E)=\sqrt{E}(\Phi_*g')(E) 
\]
for $\sigma$-a.e. $E>0$. Let us compute the right-hand side. 
Integrating by parts, we find
\[
L_\mu g'=X_\mu L_\mu g,
\]
and therefore, using Lemma~\ref{lma:g7}, 
\begin{align*}
(UL_\mu g')(E)
&=(UX_\mu L_\mu g)(E)
=
\jap{L_\mu g,(G_\mu+E)^{-1}\1_\mu}_\mu
\\
&=
\jap{g,L_\mu^*(L_\mu L_\mu^*+E)^{-1}\1_\mu}
=
\jap{g,(L_\mu^* L_\mu+E)^{-1}L_\mu^*\1_\mu}
\\
&=
\jap{g,(H+E)^{-1}h}.
\end{align*}
Putting this all together, we obtain \eqref{eq:g6a}.
\end{proof}

\subsection{Proof of Theorem~\ref{thm:b7}}
Our aim is to identify the partial isometry $\Phi_*$, defined ``abstractly'' through \eqref{eq:g5}, with the integral operator $\Phi$, defined ``concretely'' via the integral kernel $\varphi_E(t)$ in \eqref{eq:g7}. 
Since $\Phi_*$ already has all the declared properties, this will establish Theorem~\ref{thm:b7}. 

First we compare $\Phi f$ and $\Phi_*f$ for $f\in C_{\comp,0}^\infty(\bbR_+)$. Writing $f=g'$ with $g\in C_{\comp}^\infty(\bbR_+)$, by Lemma~\ref{prp:g3a} we find
\[
(\Phi_* g')(E)
=\frac1{\sqrt{E}}\jap{g,(H+E)^{-1}h}
=\frac1{\sqrt{E}}\jap{g,\varphi_E'}
=-\frac1{\sqrt{E}}\jap{g',\varphi_E}
\]
by integration by parts at the last step. 
Recalling the definition of $\Phi$, we conclude that 
\begin{equation}
\Phi_*f=\Phi f, \quad f\in C_{\comp,0}^\infty(\bbR_+).
\label{eq:g26}
\end{equation}

It remains to check \eqref{eq:g26} also for all $f\in C_{\comp}^\infty(\bbR_+)$. This point is a little subtle. Of course, we can approximate $f$ by elements of $C_{\comp,0}^\infty(\bbR_+)$ in the norm of $L^2(\bbR_+)$, but since we do not know \emph{a priori} whether $\Phi$ is bounded, we need a roundabout argument.

Let us follow the logic of the proof of Theorem~\ref{thm:b5}(ii). Recall that ${w_{\log}}H{w_{\log}}$ is trace class. On the other hand, by \eqref{eq:g6} we can write
\[
H=\Phi_*^*X_\sigma\Phi_*
\]
and therefore 
\[
{w_{\log}} H {w_{\log}}={w_{\log}}\Phi_*^*X_\sigma\Phi_* {w_{\log}}
=(X_\sigma^{1/2}\Phi_* {w_{\log}})^*(X_\sigma^{1/2}\Phi_* {w_{\log}}). 
\]
It follows that the operator $X_\sigma^{1/2}\Phi_* {w_{\log}}$ is Hilbert-Schmidt. This implies
that $\Phi_*$ is an integral operator, with the integral kernel $\Phi_*(E,t)$ satisfying 
\[
\int_0^\infty \abs{\Phi_*(E,t)}^2(w_{\log}(t))^2\dd t<\infty,\quad \text{$\sigma$-a.e. $E>0$.}
\]
Now we recall that by \eqref{eq:g26}, we have 
\begin{equation}
\varphi_E(t)=\Phi_*(E,t)+C_E, \quad \text{$\sigma$-a.e. $E>0$,}
\label{eq:g8}
\end{equation}
with some constant $C_E$. 

Let us prove that $C_E=0$ for $\sigma$-a.e. $E>0$. By the argument of the proof of Theorem~\ref{thm:b5}(ii), the kernel $\Phi_*(E,t)$ satisfies the weak Hankel equation:
\[
\int_0^\infty \Phi_*(E,t)Hf(t)\dd t=E\int_0^\infty \Phi_*(E,t)f(t)\dd t
\]
for all $f\in C_{\comp}^\infty(\bbR_+)$. 
Let us substitute \eqref{eq:g8} here and use that (by Lemma~\ref{prp:b10}(ii)) $\varphi_E$ also satisfies the weak Hankel equation \eqref{eq:b12} on functions $f$ with zero average. We obtain 
\[
C_E\int_0^\infty Hf(t)\dd t=0
\]
for all functions $f\in C_{\comp,0}^\infty(\bbR_+)$. It is easy to see that the integral here can be made non-zero by a choice of $f$, and so  $C_E=0$. 
\qed

\section{Co-finite semi-regular case}
\label{sec:i}

Throughout this section we assume that $\mu$ is Carleson and co-finite. 
In this case we have (see \eqref{eq:b7b} and \eqref{eq:b1a}) 
\begin{equation}
h(t)\leq C/t \quad\text{ and }\quad h(t)\leq C/t^2, \quad t>0. 
\label{eq:b19}
\end{equation}
In particular, $h^{\rm int}\in L^2(\bbR_+)$. 

\subsection{Some heuristics}\label{sec:i0}
Let us start with some heuristics in the spirit of Section~\ref{sec:g0}, except that now we integrate the Hankel equation instead of differentiating it. Assume that $\theta_E$ solves $H\theta_E=E\theta_E$ for some $E>0$. Denoting
\[
\Theta_E(t)=-\int_t^\infty \theta_E(s)\dd s,
\]
let us write the Hankel equation as $H\Theta_E'=E\Theta_E'$ and integrate by parts in the left-hand side:
\[
-\int_0^\infty h'(t+s)\Theta_E(s)\dd s-\Theta_E(0)h(t)=E\Theta_E'(t). 
\]
Subsequently integrating over $t$, we obtain 
\[
(H+E)\Theta_E=-\left\{\int_0^\infty \theta_E(s)\dd s\right\}h^{\rm int}.
\]
As in the finite semi-regular case, from here we learn that firstly, there is only one solution to $H\theta_E=E\theta_E$ (up to scaling) and secondly, normalising this solution by 
\begin{equation}
\int_0^\infty \theta_E(s)\dd s=-1,
\label{eq:b14a}
\end{equation}
we can find it explicitly by 
\begin{equation}
\boxed{
\theta_E(t)=\frac{\dd}{\dd t}\bigl((H+E)^{-1}h^{\rm int}\bigr)(t), \quad t>0. }
\label{eq:b14}
\end{equation}
Below we \emph{ignore} the normalisation \eqref{eq:b14a}, \emph{define} $\theta_E$ by \eqref{eq:b14} and prove that it solves $H\theta_E=E\theta_E$ for $H$-spectrally almost every $E>0$ and effects the diagonalisation of the operator $H$. Normalisation condition \eqref{eq:b14a} will then appear automatically for $H$-spectrally a.e. $E>0$. 

\subsection{The function $\theta_E$}

For any $E\in\bbC\setminus(-\spec(H))$, we define $\theta_E$ by \eqref{eq:b14}. 
\begin{lemma}\label{prp:b11}
Let $H$ be a bounded positive Hankel operator such that the measure $\mu$ is co-finite.
\begin{enumerate}[\rm (i)]
\item
For any $E\in\bbC\setminus(-\spec(H))$, the function $(H+E)^{-1}h^{\rm int}$ is in $C^\infty(\bbR_+)$ and, in particular, the derivative in \eqref{eq:b14} is well-defined. The function $\theta_E$ satisfies
\begin{align}
\abs{\theta_E(t)}&\leq C t^{-\frac{3}{2}}, \quad 0<t<1,
\label{eq:b15}
\\
\abs{\theta_E(t)}&\leq C t^{-2}, \quad t>1.
\label{eq:b16}
\end{align}
\item
The function $\theta_E$ is a weak solution of the Hankel equation $H\theta_E=E\theta_E$ in the following sense:
\begin{equation}
\int_0^\infty \theta_E(t)(Hf)(t)\dd t=E\int_0^\infty \theta_E(t)f(t) \dd t, 
\label{eq:b17}
\end{equation}
for all $f\in C_{\comp}^\infty(\bbR_+)$ satisfying the constraint
\begin{equation}
(Hf)(0)=\int_0^\infty f(t)h(t)\dd t=0.
\label{eq:b18}
\end{equation}
\item
For $E>0$, if $Hf=Ef$ with $f\in {\calB}_0$, then $f\in L^1(\bbR_+)$ and $f$ is collinear with $\theta_E$:
\begin{equation}
f(t)=-\theta_E(t)\int_0^\infty f(s)\dd s. 
\label{eq:b18a}
\end{equation}
\item
For $H$-spectrally almost every $E>0$, we have $\theta_E\in{\calB}_0$ and  $H\theta_E=E\theta_E$ and the normalisation \eqref{eq:b14a} holds. 
\end{enumerate}
\end{lemma}

\begin{remark*}
For each function $f\in C_{\comp}^\infty(\bbR_+)$ satisfying \eqref{eq:b18}, the integral on the left-hand side of \eqref{eq:b17} converges absolutely. 
\end{remark*}

\begin{remark*}
There is a subtle difference between the finite and co-finite semi-regular cases. 
The implication
\[
\theta_E\in {\calB}_0 \quad \Rightarrow \quad E\in\spec(H)
\quad\text{ is false!}
\]
Indeed, suppose $H$ is a finite rank operator. Then $h$ is a finite linear combination of exponentials, and $\theta_E$ is in ${\calB}_0$ for \emph{all} $E>0$. However, $\theta_E$ is a true eigenfunction of $H$ only at the eigenvalues of $H$, see \eqref{eq:b5}.
\end{remark*}

\subsection{The eigenfunction expansion of $H$}
For any function $f\in C^\infty_{\comp}(\bbR_+)$ and for $E>0$ we set 
\begin{equation}
\boxed{
(\Theta f)(E)=-\frac1{\sqrt{E}}\int_0^\infty \theta_{E}(t) f(t)\dd t.}
\label{eq:b20}
\end{equation}

\begin{theorem}\label{thm:b12}
Let $H$ be a bounded positive Hankel operator such that the measure $\mu$ in the Laplace transform representation \eqref{eq:a6} is co-finite. Let $\Theta$ be the integral transform defined by \eqref{eq:b20} on $f\in C^\infty_{\comp}(\bbR_+)$. Then there exists a finite measure $\rho$ on $(0,\infty)$ such that $\Theta$ extends to a partial isometry from $L^2(\bbR_+)$ onto $L^2(\rho)$, with $\Ker\Theta=\Ker H$, i.e. the Parseval identity 
\begin{equation}
\int_0^\infty \abs{(\Theta f)(E)}^2\dd\rho(E)=\norm{P_{\overline{\Ran H}}f}^2, 
\quad
f\in C^\infty_{\comp}(\bbR_+),
\label{eq:b20a}
\end{equation}
holds true, where $P_{\overline{\Ran H}}$ is the orthogonal projection onto $\overline{\Ran H}$ in $L^2(\bbR_+)$.  Thus extended, $\Theta$ intertwines $H$ and the multiplication by the independent variable in $L^2(\rho)$, i.e. 
\[
(\Theta H f)(E)=E(\Theta f)(E) \quad \text{for $\rho$-a.e. $E>0$.}
\]
\end{theorem}

Below we describe the cyclic element of $H^\perp$, formally defined by 
\[
w_\mu=H^{1/2}\1,
\]
where $\1$ is the function identically equal to $1$. 
\begin{proposition}\cite{PuTreil}\label{thm:b13}
Let $H$ be a positive Hankel operator such that the measure $\mu$ in the Laplace transform representation \eqref{eq:a6} is co-finite. Then there exists a unique element $w_\mu\in\overline{\Ran H}$ such that for any $g\in C_{\comp}^\infty(\bbR_+)$ we have 
\[
\jap{w_\mu,H^{1/2}g}
=
\int_0^\infty \int_0^\infty h(t+s)\overline{g(t)}\dd t\, \dd s. 
\]
This element is cyclic for $H^\perp$ and 
\[
\norm{w_\mu}^2=\int_0^\infty\frac{\dd\mu(\lambda)}{\lambda^2}.
\]
\end{proposition}
Below we will see that the measure $\rho$ of Theorem~\ref{thm:b12} is the spectral measure of $H$ corresponding to the element $w_\mu$. In other words, we will see that 
\[
\Theta w_\mu =1.
\]

\subsection{Example: Rosenblum's operator}

Consider \emph{Rosenblum's operator}, i.e. the Hankel operator with the kernel function 
\[
h(t)=\ee^{-t/2}/t=\int_{1/2}^\infty \ee^{-t\lambda}\dd\lambda. 
\]
This operator was diagonalised by M.~Rosenblum \cite{Ros}; see also \cite{Magnus,Ya1}. The key identity here (due to Shanker \cite{Shanker}) is
\begin{equation}
\int_0^\infty \frac{\ee^{-(t+s)/2}}{t+s}s^{-\frac12}K_{\ii k}(s/2)\dd s
=
\pi\sech(\pi k)t^{-\frac12}K_{\ii k}(t/2), \quad t>0, \quad k>0,
\label{eq:i9c}
\end{equation}
where $K_\nu$ is the modified Bessel function of the third kind (also known as Macdonald functions) \cite[Section 7.2.2]{BE}.  
(Note that $K_{-\ii k}=K_{\ii k}$, i.e. changing the sign of $k$ does not give a new solution.)

We recall that for $\Re \nu>-\tfrac12$ and $\Re z>0$ the function  $K_\nu(t)$ has the integral representation \cite[Section 7.3.4, formula (15)]{BE}
\begin{equation}
\Gamma(\nu+\tfrac12)K_\nu(z)=\sqrt{\pi}\bigl(\tfrac{z}{2}\bigr)^\nu\int_1^\infty \ee^{-zu}(u^2-1)^{-\frac12+\nu}\dd u.
\label{eq:i9a}
\end{equation}
From here it is easy to see its properties. In particular, $K_\nu$  satisfies 
 \cite[Sections 10.25, 10.30 and 10.45]{DLMF}
\begin{align}
K_\nu(t)&=O(t^{-1/2}\ee^{-t}), \quad t\to\infty,  \quad \Re \nu>0,
\notag
\\
K_\nu(t)&\sim \frac12\Gamma(\nu)(t/2)^{-\nu}, \quad t\to0, \quad \Re \nu>0,
\label{eq:i9b}
\\
K_{\ii k}(t)&=O(1), \quad t\to0_+, \quad k\in\bbR.
\notag
\end{align}
Returning to \eqref{eq:i9c}, we write it as 
\[
Hf=Ef, \quad f(t)=t^{-\frac12}K_{\ii k}(t/2), \quad E=\pi\sech(\pi k)\in (0,\pi), \quad k>0,
\]
with $f\in\calB_0$. 
By Lemma~\ref{prp:b11}(iii), the function $f$ is collinear with $\theta_E$. In order to establish the coefficient of proportionality, we need to compute the integral in \eqref{eq:b18a}. 
The identity 
\[
\int_0^\infty t^{-\frac12}K_{\ii k}(t/2)\dd t=\frac12\abs{\Gamma(\tfrac14+\tfrac{\ii}{2}k)}^2
\]
is known \cite[6.561(16)]{GR} and can be easily derived from the integral representation \eqref{eq:i9a}. From this identity we find 
\[
\theta_E(t)=\frac{2}{\abs{\Gamma(\tfrac14+\tfrac{\ii}{2}k)}^2}f(t)
=\frac{2}{\abs{\Gamma(\tfrac14+\tfrac{\ii}{2}k)}^2}t^{-1/2}K_{\ii k}(t/2). 
\]

This explicit formula for $\theta_E$ can be extended to values of $E$ outside $[0,\pi]$ by using analytic continuation:
\[
\theta_E(t)=\frac{2}{\Gamma(\frac14+\frac{\nu}{2})\Gamma(\frac14-\frac{\nu}{2})}t^{-1/2}K_\nu(t/2),
\quad
\text{ where }
E=\pi\sec(\pi\nu).
\]
Note that the right-hand side here is even in $\nu$, and therefore the choice of the branch of the inverse $\sec$ does not affect this definition. We confine ourselves to two cases:

Case 1: 
$\nu\in(0,\tfrac12)$ corresponds to \emph{positive} values $E=\pi\sech(\pi\nu)>\pi$ outside the spectrum. By \eqref{eq:i9b} we have 
\[
\theta_E(t)\sim C_E t^{-\frac12-\nu}, \quad t\to0.
\]
In particular, $\theta_E$ is not in $\calB_0$. 

Case 2: $\nu\in(\tfrac12,1)$ corresponds to \emph{negative} values $E=\pi\sech(\pi\nu)<-\pi$ outside the spectrum. In this case, again by  \eqref{eq:i9b}, the function $\theta_E$ is even ``worse'': it is not even integrable near $t=0$.

Let us discuss Parseval's identity \eqref{eq:b20a}. 
The unitarity of the Kontorovich-Lebedev transform \cite[Section 2.3]{Yakubovich} can be written as
\[
\norm{f}^2
=\frac{2}{\pi^2}\int_0^\infty k\sinh(\pi k)\Abs{\int_0^\infty t^{-\frac12}K_{\ii k}(t/2)f(t)\dd t}^2\dd k.
\]
Using $\theta_E$, we rewrite this as
\[
\norm{f}^2
=\frac{1}{2\pi^2}\int_0^\infty k\sinh(\pi k)\Abs{\Gamma(\tfrac14+\ii k)}^4
\Abs{\int_0^\infty \theta_E(t) f(t)\dd t}^2\dd k,
\]
where $E$ and $k$ related by $E=\pi\sech(\pi k)$. 
Changing integration from $\dd k$ to $\dd E$, after a little calculation this rewrites as
\begin{align*}
\norm{f}^2
&=
\frac1{2\pi^2}
\int_0^\pi\frac{k(E)}{E}\abs{\Gamma(\tfrac14+\tfrac{\ii}{2}k(E))}^4\abs{\widetilde{f}(E)}^2\dd E,
\end{align*}
where
\[
\widetilde{f}(E)=-\frac1{\sqrt{E}}
\int_0^\infty \theta_E(t) f(t)\dd t. 
\]
Thus, the corresponding spectral measure $\rho$ in this case is 
\[
\dd\rho(E)=\frac1{2\pi^2}\frac{k(E)}{E}\abs{\Gamma(\tfrac14+\tfrac{\ii}{2}k)}^4\dd E,
\]
supported on $[0,\pi]$.

\subsection{Proof of Lemma~\ref{prp:b11}}
Part (i):
we denote
\[
\Theta_E=(H+E)^{-1}h^{\rm int}\in L^2(\bbR_+)
\]
for brevity. Observe the identity
\begin{equation}
E \Theta_E= (E+H-H)(H+E)^{-1}h^{\rm int}
=
h^{\rm int}-H\Theta_E.
\label{eq:i3}
\end{equation}
We already know (see Proposition~\ref{thm:b4})  that  $h^{\rm int}$ is smooth on $\bbR_+$ and that elements of $\Ran H$ are smooth. Thus, from \eqref{eq:i3} we obtain that $\Theta_E$ is smooth. 

Let us check the estimates \eqref{eq:b15} and \eqref{eq:b16} for $\theta_E$.  
By \eqref{eq:i3} we find 
\begin{align*}
E \theta_E(t)&=\frac{\dd}{\dd t}(h^{\rm int}-H\Theta_E)
=-h(t)-\int_0^\infty h'(t+s)\Theta_E(s)\dd s.
\end{align*}
From here  by Cauchy-Schwarz, using \eqref{eq:b7b} and \eqref{eq:b1a},  we obtain 
\begin{align*}
\abs{E \theta_E(t)+h(t)}&\leq C t^{-\frac32}, \quad 0<t<1,
\\
\abs{E \theta_E(t)+h(t)}&\leq C t^{-\frac52}, \quad t>1.
\end{align*}
We already know that $h$ satisfies the estimates \eqref{eq:b19}. Thus, we obtain the estimates \eqref{eq:b15} and \eqref{eq:b16} for $\theta_E$.

Part (ii):
Let us check the weak eigenvalue equation \eqref{eq:b17}. 
We first observe that from \eqref{eq:i3}, using Cauchy-Schwarz, it is easy to obtain 
\begin{align}
\abs{\Theta_E(t)}&\leq Ct^{-\frac12}, \quad 0<t<1,
\label{eq:i4}
\\
\abs{\Theta_E(t)}&\leq Ct^{-\frac32}, \quad t>1.
\label{eq:i5}
\end{align}
Integrating by parts, we find
\[
\jap{\theta_E,Hf}
=
\jap{\Theta_E',Hf}
=
-\jap{\Theta_E,(Hf)'}.
\]
The boundary term $\lim_{t\to0}\Theta_E(t) Hf(t)$  disappears because of the assumption $Hf(0)=0$ and \eqref{eq:i4}; the boundary term $\lim_{t\to\infty}\Theta_E(t)Hf(t)$ disappears because of \eqref{eq:i5} and because $Hf(t)=O(t^{-2})$ as $t\to\infty$. Further, integrating by parts, we find
\[
(Hf)'(t)
=\int_0^\infty h'(t+s)f(s)\dd s
=-(Hf')(t). 
\]
Thus, we find
\begin{align*}
-\jap{\Theta_E,(Hf)'}
&=\jap{\Theta_E,Hf'}
=\jap{(H+E)^{-1}h^{\rm int},Hf'}
\\
&=\jap{H(H+E)^{-1}h^{\rm int},f'}
=\jap{h^{\rm int},f'}-E\jap{\Theta_E,f'}.
\end{align*}
The term $\jap{h^{\rm int},f'}$ vanishes again by assumption $Hf(0)=0$. Integrating by parts once again, we arrive at the right-hand side of \eqref{eq:b17}.

Part (iii): 
Let $E>0$ and let $Hf=Ef$ with $f\in {\calB}_0$ non-zero. We need to justify the reasoning of Section~\ref{sec:i0}. 
By \eqref{eq:b1a} and our assumption $f\in {\calB}_0$, we have, exactly as in \eqref{eq:e3a}, 
\[
\abs{f(t)}\leq C
\left( 
\int_0^1 (s+t)^{-2}s^{-\frac12-\eps}\dd s
+ \int_1^\infty (s+t)^{-2}s^{-\frac12+\eps}\dd s
\right)
=O(t^{-\frac{3}{2}+\eps}), \quad t\to\infty. 
\]
Bootstrapping this estimate, we find 
\[
\abs{f(t)}\leq O(t^{-2})+C\int_1^\infty (s+t)^{-2}s^{-\frac32+\eps}\dd s=O(t^{-2}), \quad t\to\infty.
\]
Denoting 
\[
F(t)=-\int_t^\infty f(s)\dd s,
\]
we conclude that $F\in L^2(\bbR_+)$. 

Now the integration by parts argument of Section~\ref{sec:i0} can be justified, which gives 
\begin{equation}
(H+E)F=F(0)h^{\rm int}. 
\label{eq:i7}
\end{equation}
If $F(0)=0$, then we get $F=0$ and $f=0$, which is a contradiction. Thus, $F(0)\not=0$, and then  \eqref{eq:i7} yields $F=F(0)(H+E)^{-1}h^{\rm int}$, i.e. \eqref{eq:b18a}, as claimed.

Part (iv): by Theorem~\ref{thm:b5}(ii), for $H$-spectrally almost every $E>0$, there is $f\in {\calB}_0$ with $Hf=Ef$. By part (ii), $f$ is collinear with $\theta_E$ and so $H\theta_E=E\theta_E$. Substituting $f=\theta_E$ in \eqref{eq:b18a}, we obtain the normalisation condition \eqref{eq:b14a}. 

The proof of Lemma~\ref{prp:b11} is complete. \qed

\subsection{The unitary map $V$}
In the rest of this section we prove Theorem~\ref{thm:b12}. 
As in the previous section, we will work with the space $L^2(\mu)$ ($\mu$ is a \emph{co-finite} Carleson measure, as in Theorem~\ref{thm:b12})  and use the results of \cite{PuTreil}. Notation $\omega_{\mu}$ stands for the function $\omega_{\mu}(x)=1/x$ on $\bbR_+$, considered as an element of $L^2(\mu)$.

\begin{theorem}\label{thm:i1}
Let $\mu$ be a co-finite Carleson measure on $\bbR_+$. Then the element $\omega_{\mu}$ is cyclic for the operator $G_{\mu}$. Denote by $\rho$ the spectral measure for $G_{\mu}$ corresponding to the element $\omega_{\mu}$: 
\[
\rho(\Delta)=\jap{{\mathcal E}_{G_\mu}(\Delta)\omega_{\mu},\omega_{\mu}}_{\mu},
\quad \Delta\subset\bbR.
\]
Then there exists a unitary operator $V: L^2(\mu)\to L^2(\rho)$ such that 
\begin{equation}
VG_{\mu}=X_\rho V, \quad 
VX_{\mu}^{-1}=G_\rho V, \quad
V\omega_{\mu}=\1_\rho.
\label{eq:i1}
\end{equation}
\end{theorem}
NB: note the inverse in the second formula in \eqref{eq:i1}, compare with \eqref{eq:g7a}.

This theorem is implicit in \cite{PuTreil}; in the rest of this subsection we explain the proof.

Let us define another measure $\mu^{\#}$ on $\bbR_+$ by 
\[
\int_0^\infty \psi(x)\dd\mu^{\#}(x)
=
\int_0^\infty \psi(1/y)\frac{\dd\mu(y)}{y^2}, 
\quad \psi\in C_{\comp}^\infty(\bbR_+).
\]
It is clear that the map $\mu\mapsto \mu^{\#}$ is an involution on measures. It is also clear that $\mu$ is co-finite if and only if $\mu^{\#}$ is finite and it is easy to check that $\mu^{\#}$ is Carleson if and only if $\mu$ is Carleson. 

Consider the unitary operator
\[
J:L^2(\mu)\to L^2(\mu^{\#}), \quad (Jf)(x)=\frac1xf\left(\frac1x\right).
\]
A direct calculation shows that 
\begin{equation}
G_{\mu^{\#}}J=JG_{\mu}, 
\qquad
JX_{\mu}^{-1}=X_{\mu^{\#}} J,
\qquad
J\omega_{\mu}=\1_{\mu^{\#}}.
\label{eq:i6}
\end{equation}

We will use the involution $\mu\mapsto \mu^{\#}$ in order to deduce Theorem~\ref{thm:i1} (co-finite measures) from Proposition~\ref{thm:PuTr-2} (finite measures). 

We need a word about notation. In Proposition~\ref{thm:PuTr-2}, the operator $U$ and the measure $\sigma$ arise from a finite measure $\mu$. To make our argument clearer, we shall indicate the dependence of $U$ and $\sigma$ on the measure $\mu$ by writing $U_\sigma$ and $\sigma_\mu$.  

\begin{proof}[Proof of Theorem~\ref{thm:i1}]
Let $\mu$ be co-finite and let $\mu^{\#}$ be its finite counterpart. 
By Proposition~\ref{thm:PuTr-2}, there exists a spectral measure $\sigma$ of $G_{\mu^{\#}}$ and a unitary $U: L^2(\mu^{\#})\to L^2(\sigma)$ satisfying the properties \eqref{eq:g7a}, which we write explicitly as
\begin{equation}
UG_{\mu^{\#}}=X_\sigma U, \quad 
UX_{\mu^{\#}}=G_\sigma U, \quad
U\1_{\mu^{\#}}=\1_\sigma.
\label{eq:i6a}
\end{equation}
We set 
\[
\rho=\sigma\quad\text{ and }\quad V=UJ.
\] 
Using \eqref{eq:i6} and \eqref{eq:i6a}, it is immediate to check the requirements \eqref{eq:i1}. 
\end{proof}

\subsection{The unitary map $\Theta_*$}
As in the previous section, we work with the operators $H$ in $L^2(\bbR_+)$, $G_\mu$ in $L^2(\mu)$ and the Laplace transform $L_\mu: L^2(\bbR_+)\to L^2(\mu)$, with 
\[
H=L_\mu^*L_\mu
\quad \text{ and }\quad
G_\mu=L_\mu L_\mu^*.
\]
The polar decomposition of $L_\mu^*$ can be written as
\[
L_\mu^*=T_\mu G_\mu^{1/2}, \quad L_\mu=G_\mu^{1/2}T_\mu^*,
\]
where $T_\mu$ is an isometry from $L^2(\mu)$ onto $\overline{\Ran H}$. 
Let $V:L^2(\mu)\to L^2(\rho)$ be the operator from Theorem~\ref{thm:i1}. 
We define the partial isometry
\[
\Theta_*: L^2(\bbR_+)\to L^2(\rho)
\quad\text{ by }\quad
\Theta_*=VT_\mu^*.
\]
From the definition, we see that $\Ker\Theta_*=\Ker H$ and $\Ran\Theta_*=L^2(\rho)$. 
Our aim is to prove that $\Theta=\Theta_*$. 

We recall that the element $w_\mu\in\overline{\Ran H}$ is defined in Theorem~\ref{thm:b13}. 
\begin{proposition}\label{prp:i2}
The partial isometry $\Theta_*$ satisfies:
\[
\Theta_*H=X_\rho\Theta_*, \quad \Theta_*w_\mu=\1_\rho. 
\]
\end{proposition}
\begin{proof}
As in the proof of Lemma~\ref{prp:g3}, from 
$T_{\mu}^*H=G_{\mu}T_{\mu}^*$
and $VG_\mu=X_\rho V$ we obtain the first property $\Theta_*H=X_\rho\Theta_*$. 

The identity
\[
T_\mu\omega_\mu=w_\mu
\]
was proved in \cite[Lemma~6.2]{PuTreil}. It implies that $\omega_\mu=T_\mu^* w_\mu$. From here and $V\omega_\mu=\1_\rho$ we obtain the second property $\Theta_*\omega_\mu=\1_\rho$. 
\end{proof}

\subsection{Proof of Theorem~\ref{thm:b12}}
Our aim is to prove that $\Theta f=\Theta_*f$ for all $f\in C_{\comp}^\infty(\bbR)$. Since $\Theta_*$ has all the declared properties, this will establish Theorem~\ref{thm:b12}. 

As in \eqref{eq:g6a}, we have 
\[
X_\rho^{1/2}\Theta_*
=X_\rho^{1/2}VT_{\mu}^*
=VG_{\mu}^{1/2}T_{\mu}^*
=VL_{\mu}.
\]
Let us compute both sides of this identity on an element $f\in C_{\comp}^\infty(\bbR)$. 
The left-hand side is 
\[
(X_\rho^{1/2}\Theta_*f)(E)=\sqrt{E}(\Theta_*f)(E)
\]
for $\rho$-a.e. $E>0$. Let us compute the right-hand side. We have 
\[
L_\mu f=X_{\mu}^{-1}L_\mu f',
\]
and therefore
\[
VL_{\mu}f=VX_{\mu}^{-1}L_\mu f'. 
\]
Denote for brevity $g=L_\mu f'\in\Ran X_{\mu}$. 
Using \eqref{eq:i1} and similarly to the proof of Lemma~\ref{lma:g7}, we find for $\rho$-a.e. $E>0$:
\begin{align*}
(VX_\mu^{-1} g)(E)
&=(G_\rho Vg)(E)
=\jap{(X_\rho+E)^{-1}Vg,\1_\rho}_{\rho}
\\
&=\jap{V(G_\mu+E)^{-1}g,V\omega_\mu}_{\rho}
=\jap{(G_\mu+E)^{-1}g,\omega_\mu}_{\mu}
\\
&=\jap{g,(G_\mu+E)^{-1}\omega_\mu}.
\end{align*}

Using this, we find
\begin{align*}
(VL_\mu f)(E)&=(VX_{\mu}^{-1}L_\mu f')(E)
=
\jap{L_\mu f',(G_\mu+E)^{-1}\omega_\mu}_\mu
\\
&=
\jap{f',L_\mu^*(G_\mu+E)^{-1}\omega_\mu}
=
\jap{f',(H+E)^{-1}L_\mu^*\omega_\mu}
\\
&=
\jap{f',(H+E)^{-1}h^{\rm int}}
=
-\jap{f,\theta_E}. 
\end{align*}

Putting this together, we find
\[
\sqrt{E}(\Theta_*f)(E)
=
X_\sigma^{1/2}\Theta_*f(E)
=
-\jap{f,\theta_E}.
\]
This proves that $\Theta f=\Theta_*f$. 
The proof of Theorem~\ref{thm:b12} is complete. \qed

\section{Regular case}\label{sec:c}

\subsection{Solutions $\phi_E$ and $\theta_E$}
Throughout this section, we assume that $\mu$ is both finite and co-finite. 
From \eqref{eq:b1} and \eqref{eq:b1a} we find that both $h$ and  $h^{\rm int}$ belong to $L^2(\bbR_+)$.  Moreover:
\begin{lemma}
\label{lma:b0}
The elements $h$ and $h^{\rm int}$ are cyclic for $H^\perp$. 
\end{lemma}
This result is contained in \cite{PuTreil}, but in Section~\ref{sec:c3} we recall a direct proof.

For $E\in\bbC\setminus(-\spec(H))$, we define the functions $\phi_E$ and $\theta_E$ on $\bbR_+$ by
\begin{align}
\boxed{\phi_E=-\int_t^\infty \phi_E'(s)\dd s, \quad \phi_E'=(H+E)^{-1}h,
}
\label{eq:b3}
\end{align}
and
\begin{align}
\boxed{\theta_E=
\theta_E(t)=\frac{\dd}{\dd t}\bigl((H+E)^{-1}h^{\rm int}\bigr)(t), \quad t>0.
}
\label{eq:b4}
\end{align}
Here \eqref{eq:b4} is exactly the same definition as \eqref{eq:b14} in the co-finite semi-regular case, while $\phi_E$ differs from $\varphi_E$ of \eqref{eq:b13} by an additive constant:
\begin{equation}
\varphi_E=\phi_E+\int_0^\infty \phi_E'(s)\dd s-1=\phi_E-(\phi_E(0)+1). 
\label{eq:b4d}
\end{equation}
We will come back to \eqref{eq:b4d} shortly. 

\begin{lemma}\label{lma:b5}
For any $E\in\bbC\setminus(-\spec(H))$, the functions $\phi_E$ and $\theta_E$ are in $L^2(\bbR_+)$. Moreover, they are analytic in $E\in\bbC\setminus(-\spec(H))$. The limit $\phi_E(0)$ exists and is finite. 
\end{lemma}

Let us now comment on \eqref{eq:b4d}. In the finite semi-regular case, we were forced to integrate from $0$ to $t$ in the definition \eqref{eq:b13} of $\varphi_E$, because for a general value of $E$ the integral $\int_t^\infty \varphi_E'(s)\dd s$ may diverge (see e.g. \eqref{eq:g28}). On the other hand, in the regular case framework of this section, the integral in \eqref{eq:b3} does converge and defines a function $\phi_E$ that is in $L^2(\bbR_+)$ for all $E$. 

We will see shortly that both $\phi_E$ and $\theta_E$ solve the Hankel equation up to a rank-one term. But first we need to introduce the function that enters this rank-one term.

\subsection{The function $\mySigma(E)$}
Below we denote by ${\jap{h}}$ the constant that has different representations, 
\begin{equation}
{\jap{h}}=\int_0^\infty h(t)\dd t=h^{\rm int}(0)=\int_0^\infty \frac{\dd\mu(\lambda)}{\lambda}=2\Tr H.
\label{eq:b4c}
\end{equation}

We introduce an analytic function $\mySigma(E)$, $E\in\bbC\setminus(-\spec(H))$, that plays an important role and arises in several different ways:

\emph{Through the resolvent of $H$ at $-E$:}
\begin{equation}
\boxed{\mySigma(E)=1+\frac1E\bigl(-{\jap{h}}+\jap{(H+E)^{-1}h,h^{\rm int}}\bigr);}
\label{eq:b4a}
\end{equation}

\emph{Through the solutions $\phi_E$ and $\theta_E$:}
\begin{equation}
\boxed{
\begin{aligned}
\mySigma(E)&=1+\phi_E(0),
\\
\mySigma(E)&=1+\int_0^\infty \theta_E(t)\dd t;
\end{aligned}
}
\label{eq:b4ac}
\end{equation}

\emph{Through Fredholm determinants:}
\begin{equation}
\boxed{\mySigma(E)=\frac{\det(I-\tfrac1E H)}{\det(I+\tfrac1E H)}=
\prod_n \frac{1-\tfrac1E E_n}{1+\tfrac1EE_n},}
\label{eq:b4ab}
\end{equation}
where $\{E_n\}$ is the (finite or infinite) sequence of all positive eigenvalues of $H$.

Of course, the fact that all these expressions for $\mySigma(E)$ coincide requires proof. In Section~\ref{sec:c4}, we will prove

\begin{theorem}\label{thm:b1a}
For any $E\in\bbC\setminus(-\spec(H))$, the right-hand sides in \eqref{eq:b4a}, \eqref{eq:b4ac} and \eqref{eq:b4ab} coincide. 
\end{theorem}

\begin{remark*}
Coming back to \eqref{eq:b4d} and using the first equation in \eqref{eq:b4ac}, we find that 
\[
\varphi_E=\phi_E-\mySigma(E).
\]
Since by \eqref{eq:b4ab} the function $\mySigma(E)$ vanishes on the spectrum, we find 
\[
\varphi_{E_n}=\phi_{E_n}
\]
at all positive eigenvalues $E_n$ of $H$. 
\end{remark*}

\subsection{Main result in the regular case}

\begin{theorem}\label{thm:b2}
Let $H\not=0$ be a positive Hankel operator such that $\mu$ is both finite and co-finite. 
\begin{enumerate}[\rm (i)]
\item
For $E\in \bbC\setminus(-\spec(H))$, we have 
\begin{equation}
\boxed{
\begin{aligned}
(H-E)\phi_E&=\mySigma(E)h^{\rm int},
\\
(H-E)\theta_E&=\mySigma(E)h.
\end{aligned}}
\label{eq:b5} 
\end{equation}
In particular, if $E>0$ is an eigenvalue of $H$, then both $\phi_E$ and $\theta_E$ are collinear with the corresponding eigenvector. 
\item
For $E$ not in the spectrum of $H$, the elements $\phi_E$ and $\theta_E$ are linearly independent, unless $\rank H=1$. 
\end{enumerate}
\end{theorem}

Let us rephrase the theorem in the way suggestive of the analogy with the Sturm-Liouville problem, see \eqref{eq:a7a}. 

\begin{corollary}
Under the hypothesis of Theorem~\ref{thm:b2}, for any $E>0$ the following are equivalent:
\begin{enumerate}[\rm (i)]
\item
$E$ is an eigenvalue of $H$;
\item
$\phi_E$ satisfies $\phi_E(0)=-1$;
\item
$\theta_E$ satisfies $\int_0^\infty \theta_E(t)\dd t=-1$.
\end{enumerate}
\end{corollary}

In the rest of this section, we prove Lemma~\ref{lma:b0}, 
Lemma~\ref{lma:b5}, Theorem~\ref{thm:b1a} and finally the main result Theorem~\ref{thm:b2}. 

\subsection{Notation $D$ and $D^{-1}$. Proof of Lemma~\ref{lma:b5}}
For a differentiable function $f$ on $[0,\infty)$ with a sufficiently fast rate of decay as $t\to\infty$, we denote 
\[
Df=f'
\quad\text{ and }\quad
(D^{-1}f)(t)=-\int_t^\infty f(s)\dd s.
\]
Of course, the operators $D$ and $D^{-1}$ are not well-defined on the whole of $L^2(\bbR_+)$, but they are well-defined on the range of $H$. In fact, the products $DH$ and $D^{-1}H$ are explicit Hankel integral operators
\begin{align}
(DHf)(t)&=\int_0^\infty h'(t+s)f(s)\dd s, 
\label{eq:b2c}
\\
(D^{-1}Hf)(t)&=-\int_0^\infty h^{\rm int}(t+s)f(s)\dd s.
\label{eq:b2d}
\end{align}
By \eqref{eq:b1} and \eqref{eq:b1b}, both operators are bounded.

Using this notation, let us rewrite the definitions \eqref{eq:b3} and \eqref{eq:b4} as follows. We have 
\[
\phi_E'=(H+E)^{-1}h
=\frac1E(E+H-H)(H+E)^{-1}h
=\frac1E(h-H(H+E)^{-1}h),
\]
and therefore 
\begin{equation}
\phi_E=\frac1E\bigl(-h^{\rm int}-(D^{-1}H)(H+E)^{-1}h\bigr).
\label{eq:c35}
\end{equation}
Similarly, 
\[
(H+E)^{-1}h^{\rm int}
=\frac1E(E+H-H)(H+E)^{-1}h^{\rm int}
=\frac1E(h^{\rm int}-H(H+E)^{-1}h^{\rm int}),
\]
and therefore 
\begin{equation}
\theta_E=\frac1E\bigl(-h-(DH)(H+E)^{-1}h^{\rm int}\bigr).
\label{eq:c36}
\end{equation}
\begin{proof}[Proof of Lemma~\ref{lma:b5}]
From \eqref{eq:c35} and \eqref{eq:c36} it is immediate that $\phi_E$ and $\theta_E$ belong to $L^2(\bbR_+)$ and are analytic in $E$. Moreover, since $\phi_E'$ is also in $L^2(\bbR_+)$, the limit $\phi_E(0)$ is well-defined and finite. 
\end{proof}
\subsection{Lyapunov identities. Cyclicity of $h$ and $h^{\rm int}$}
\label{sec:c3}
For a differentiable function $f$ on $[0,\infty)$ with a sufficiently fast rate of decay as $t\to\infty$, integrating by parts, we find 
\[
\frac{\dd}{\dd t}\int_0^\infty h(t+s)f(s)\dd s
+
\int_0^\infty h(t+s)f'(s)\dd s
=
-f(0)h(t),
\]
which we write as
\begin{equation}
HDf+DHf=-f(0)h. 
\label{eq:c33}
\end{equation}
Similarly, we find 
\begin{align*}
\int_t^\infty\left\{\int_0^\infty h(s+u)f(u)\dd u\right\}\dd s
+&\int_0^\infty h(t+u)\left\{\int_u^\infty f(s)\dd s\right\}\dd u
\\
&=
\left(\int_0^\infty f(s)\dd s\right)h^{\rm int}(t)
\end{align*}
which we rewrite as
\begin{equation}
HD^{-1}f+D^{-1}Hf=-\left(\int_0^\infty f(s)\dd s\right)h^{\rm int}.
\label{eq:c34}
\end{equation}
Applying these identities to functions $f=Hg$, we obtain the pair of Lyapunov identities
\begin{equation}
\boxed{
\begin{aligned}
H(DH)+(DH)H&=-\jap{\cdot,h}h,
\\
H(D^{-1}H)+(D^{-1}H)H&=-\jap{\cdot,h^{\rm int}}h^{\rm int}.
\end{aligned}}
\label{eq:c1a}
\end{equation}
Variants of these identities are known in control theory \cite{Ober,MPT}.
As an easy consequence of \eqref{eq:c1a}, we immediately obtain the proof of the cyclicity of $h$ and $h^{\rm int}$:

\begin{proof}[Proof of Lemma~\ref{lma:b0}]
First let us check that $h$ is cyclic.  Suppose, to get a contradiction, that $f$ is a non-zero eigenvector of $H^\perp$ such that $f$ is orthogonal to $h$. Write $Hf=Ef$ with $E>0$, and apply the first Lyapunov identity \eqref{eq:c1a} to $f$. The right-hand side vanishes, and we obtain 
\[
(H+E)(DH)f=0,
\]
i.e. $(DH)f$ is an eigenvector of $H$ with an eigenvalue $-E$. Since $H\geq0$, this implies $DHf=0$, hence $Hf=0$, which is a contradiction.

The proof of cyclicity of $h^{\rm int}$ proceeds in the same way by using the second Lyapunov identity of \eqref{eq:c1a} instead of the first one. 
\end{proof}

\subsection{The function $\mySigma(E)$}\label{sec:c4}

Here we prove Theorem~\ref{thm:b1a}. To make the logic clear, we shall regard \eqref{eq:b4a} as the definition of $\mySigma(E)$, and \eqref{eq:b4ac} and \eqref{eq:b4ab} as the statements to be proved. 

To make our notation more concise, for $E\in\bbC\setminus(-\spec(H))$ let us denote
\[
a(E)=\jap{(H+E)^{-1}h,h^{\rm int}}.
\]
With this notation, our definition \eqref{eq:b4a} of $\mySigma(E)$ becomes 
\begin{equation}
\mySigma(E)=1+\frac1E(-{\jap{h}}+a(E)).
\label{eq:c13a}
\end{equation}

Since the integral kernel of $H$ is real-valued and the functions $h$ and $h^{\rm int}$ are real, we find that $a(E)$ is real on the real axis. This implies that $a(E)=\overline{a(\overline{E})}$ in the complex plane, or equivalently 
\begin{equation}
a(E)=\jap{(H+E)^{-1}h^{\rm int},h}.
\label{eq:c13}
\end{equation}

\begin{lemma}
Let $\mySigma(E)$ be defined by \eqref{eq:c13a}; then the identities \eqref{eq:b4ac} hold true for all $E\in\bbC\setminus(-\spec(H))$.
\end{lemma}
\begin{proof}
From \eqref{eq:b2d} it is clear that for any $f\in L^2(\bbR_+)$, 
\[
(D^{-1}Hf)(0)=-\jap{f,h^{\rm int}}. 
\]
Using this and \eqref{eq:c35}, we find
\[
1+\phi_E(0)=1+\frac1E\left(-h^{\rm int}(0)+\jap{(H+E)^{-1}h,h^{\rm int}}\right)=\mySigma(E).
\]
Similarly, from \eqref{eq:b2c} we see that for any $f\in L^2(\bbR_+)$, 
\[
\int_0^\infty (DHf)(t)\dd t=-\jap{f,h},
\]
and therefore, by \eqref{eq:c36}, 
\[
1+\int_0^\infty \theta_E(t)\dd t
=1+\frac1E\left(-\int_0^\infty h(t)\dd t+\jap{(H+E)^{-1}h^{\rm int},h}\right)=\mySigma(E).
\]
The proof is complete. 
\end{proof}

For the proof of \eqref{eq:b4ab}, we need an important auxiliary identity. 

\begin{lemma}\label{lma:b1}
The identity 
\begin{equation}
\boxed{\mySigma(E)\mySigma(-E)=1 }
\label{eq:c15}
\end{equation}
holds true whenever $E\notin\spec(H)$ and $-E\notin\spec(H)$. 
\end{lemma}
\begin{proof}
We prepare two identities. 
Applying \eqref{eq:c33} to $f=h^{\rm int}$, we find 
\begin{equation}
(DH)h^{\rm int}=Hh-{\jap{h}}h.
\label{eq:c27}
\end{equation}
Next, since $(h^{\rm int})'=-h$, we have
\[
\jap{h,h^{\rm int}}=\jap{h^{\rm int},h}=
-\int_0^\infty (h^{\rm int})'(t)h^{\rm int}(t)\dd t=\frac12 {\jap{h}}^2.
\]
Now we embark on the proof of \eqref{eq:c15}. 
Adding and subtracting $E(DH)$ in the left-hand side of the first Lyapunov identity  \eqref{eq:c1a}, we find 
\[
(H+E)(DH)+(DH)(H-E)=-\jap{\cdot,h}h.
\]
Multiplying by $(H+E)^{-1}$ on the left and by $(H-E)^{-1}$ on the right, we get
\[
DH(H-E)^{-1}+(H+E)^{-1}DH=-\jap{(H-E)^{-1}\cdot,h}(H+E)^{-1}h.
\]
Evaluating the quadratic form of both sides on $h^{\rm int}$,  we find
\begin{align}
\jap{DH(H-E)^{-1}h^{\rm int},h^{\rm int}}
&+
\jap{(H+E)^{-1}DH h^{\rm int},h^{\rm int}}
\notag
\\
&=
-\jap{(H-E)^{-1}h^{\rm int},h}\jap{(H+E)^{-1}h,h^{\rm int}}.
\label{eq:c29}
\end{align}
Consider the first term on the left-hand side here. 
Since $DH$ is an integral operator with a real symmetric kernel, it is self-adjoint. 
Using this fact and \eqref{eq:c27}, we find 
\begin{align*}
\jap{DH(H-E)^{-1}h^{\rm int},h^{\rm int}}
&=
\jap{(H-E)^{-1}h^{\rm int},(DH)h^{\rm int}}
\\
&=\jap{(H-E)^{-1}h^{\rm int},Hh}
-
{\jap{h}}
\jap{(H-E)^{-1}h^{\rm int},h}.
\end{align*}
Writing $Hh=(H-E)h+Eh$ here and recalling the definition of $a(E)$, we find
\begin{align*}
\jap{(H-E)^{-1}h^{\rm int},Hh}
-
{\jap{h}}
\jap{(H-E)^{-1}h^{\rm int},h}
&=\jap{h^{\rm int},h}+Ea(-E)-{\jap{h}}a(-E)
\\
&=\frac12{\jap{h}}^2+Ea(-E)-{\jap{h}}a(-E).
\end{align*}
Similarly, we transform the second term on the left-hand side of \eqref{eq:c29} as follows:
\begin{align*}
\jap{(H+E)^{-1}DH h^{\rm int},h^{\rm int}}
&=
\jap{(H+E)^{-1}Hh,h^{\rm int}}
-
{\jap{h}}\jap{(H+E)^{-1}h,h^{\rm int}}
\\
&=\frac12{\jap{h}}^2-Ea(E)-{\jap{h}}a(E).
\end{align*}
Substituting this into \eqref{eq:c29}, we find
\begin{align*}
{\jap{h}}^2+Ea(-E)-{\jap{h}}a(-E)
-Ea(E)-{\jap{h}}a(E)
=
-a(-E)a(E).
\end{align*}
Recalling \eqref{eq:c13a}, after elementary algebra this rewrites as  \eqref{eq:c15}.
\end{proof}

\begin{lemma}
Identity \eqref{eq:b4ab} holds true. 
\end{lemma}
\begin{proof}
Let us denote 
\begin{equation}
\mySigma_{\rm Blaschke}(E)=\prod_n \frac{1-\frac1E E_n}{1+\frac1E E_n}; 
\label{eq:c13b}
\end{equation}
we need to prove that 
\[
\mySigma(E)=\mySigma_{\rm Blaschke}(E)
\]
for all $E\in\bbC\setminus(-\spec(H))$.
By the definition \eqref{eq:c13}, \eqref{eq:c13a}, the function $\mySigma(E)$ is analytic in the right half-plane $\Re E>0$. Let us discuss its analytic structure. 

\emph{Zeros of $\mySigma$:}
since both $h$ and $h^{\rm int}$ are cyclic for $H^\perp$, we see that $a(-E)$ has poles at every positive eigenvalue $E_n$ of $H$ (i.e. there cannot be any cancellation in the inner product), and all these poles are simple. By the identity \eqref{eq:c15}, it follows that $\mySigma(E)$ has a simple zero at every positive eigenvalue of $H$. There are no other zeros of $\mySigma(E)$ in the open right half-plane. Thus, $\mySigma_{\rm Blaschke}$ is exactly the Blaschke product constructed from the zeros of $\mySigma$.

\emph{Canonical factorisation:}
the function $a(E)$ satisfies
\begin{equation}
\abs{a(E)}\leq C\norm{(H+E)^{-1}}\leq C/E, \quad \Re E>0, 
\label{eq:c31}
\end{equation}
and therefore $\mySigma(E)$ satisfies
\[
\abs{\mySigma(E)}\leq 1+C\abs{E}^{-2}, \quad \Re E>0.
\]
Thus, $\mySigma(E)$ is a function of the \emph{Smirnov class} 
\cite[Chapter 3]{Nikolski} in the right half-plane. 
As such, $\mySigma(E)$ can be written as a \emph{canonical product}
\[
\mySigma=\gamma \mySigma_{\rm out}\mySigma_{\rm Blaschke}\mySigma_{\rm sing},
\]
where $\gamma$ is a unimodular complex number,  $\mySigma_{\rm out}$ is the outer factor of $\mySigma$, 
\[
\mySigma_{\rm out}(E)=\exp\left\{\frac{\ii }{\pi}\int_{-\infty}^\infty \left(\frac1{\ii E-t}+\frac{t}{t^2+1}\right)\log\abs{\mySigma(\ii t)}\dd t\right\}, 
\quad \Re E>0,
\]
$\mySigma_{\rm Blaschke}$ is the Blaschke product \eqref{eq:c13b} over all zeros of $\mySigma$ in the right half-plane, and $\mySigma_{\rm sing}$ is a singular inner function, 
\[
\mySigma_{\rm sing}(E)=\exp\left\{-\frac{\ii }{\pi}\int_{-\infty}^\infty \left(\frac1{\ii E-t}+\frac{t}{t^2+1}\right)\dd \nu(t)\right\}\ee^{-\alpha E}, 
\quad \Re E>0,
\]
where $\nu$ is a singular measure on $\bbR$ and $\alpha\geq0$. 
Let us discuss these factors in turn. 

\emph{The outer factor $\mySigma_{\rm out}$:}
as already discussed,  $\mySigma(E)$ is real on the real line and therefore satisfies $\overline{\mySigma(E)}=\mySigma(\overline{E})$ in the complex plane. Thus, identity \eqref{eq:c15} implies that 
\[
\abs{\mySigma(\ii E)}=1, \quad E\in\bbR\setminus\{0\}.
\]
This means that the outer factor $\mySigma_{\rm out}(E)$ in the canonical factorisation equals $1$. In other words, $\mySigma(E)$ is an inner function in the right half-plane.

\emph{The singular inner factor $\mySigma_{\rm sing}$:}
we have established that the canonical factorisation for $\mySigma$ reduces to 
\[
\mySigma=\gamma \mySigma_{\rm Blaschke} \mySigma_{\rm sing}.
\]
From the definition it is evident that $\mySigma(E)$ is analytic in the neighbourhood of every point on the imaginary axis except possibly the origin. Thus, the singular measure $\nu$ must be concentrated at $0$. In other words, we have
\begin{equation}
\mySigma(E)=\gamma \mySigma_{\rm Blaschke}(E)\ee^{-\alpha E}\ee^{-\beta/E} 
\label{eq:c32}
\end{equation}
with some $\abs{\gamma}=1$, $\alpha\geq0$ and $\beta=\nu(\{0\})/\pi\geq0$. Our task is to show that $\gamma=1$ and $\alpha=\beta=0$. 

By \eqref{eq:c31} we find that 
\[
\mySigma(E)=1-\frac{{\jap{h}}}{E}+O(E^{-2}), \quad \abs{E}\to\infty. 
\]
On the other hand, the Blaschke product $\mySigma_{\rm Blaschke}(E)$ has the asymptotics
\[
\mySigma_{\rm Blaschke}(E)=1-2\frac{\sum_n E_n}{E}+O(E^{-2})=1-2\frac{\Tr H}{E}+O(E^{-2}), \quad \abs{E}\to\infty. 
\]
Substituting this into \eqref{eq:c32}, we immediately find that $\gamma=1$ and $\alpha=0$. Next, comparing the $O(1/E)$ terms on both sides, we find
\[
{\jap{h}}=2\Tr H+\beta,
\]
which forces $\beta=0$. 
\end{proof}

\subsection{Proof of Theorem~\ref{thm:b2}}
(i) 
Let us prove the identities \eqref{eq:b5}; this is the core of the proof. 
It is convenient to start from the second identity. We use notation ${\jap{h}}$, see \eqref{eq:b4c}. 
Let us apply $H$ to the formula \eqref{eq:c36} for $\theta_E$ and use the first Lyapunov identity \eqref{eq:c1a}:
\begin{align*}
H\theta_E
&=\frac1E\bigl(-Hh-H(DH)(H+E)^{-1}h^{\rm int}\bigr)
\\
&=\frac1E\bigl(-Hh+(DH)H(H+E)^{-1}h^{\rm int}+\jap{(H+E)^{-1}h^{\rm int},h}h \bigr).
\end{align*}
Writing $(DH)H=(DH)(H+E-E)$ here, we find 
\begin{align*}
H\theta_E
=\frac1E\bigl(-Hh+(DH)h^{\rm int}-E(DH)(H+E)^{-1}h^{\rm int}+a(E)h \bigr).
\end{align*}
Using \eqref{eq:c27}, we rewrite this as
\begin{align*}
H\theta_E
=
\frac1E\bigl (-{\jap{h}} h-E(DH)(H+E)^{-1}h^{\rm int}+a(E)h \bigr)
=E\theta_E+\mySigma(E)h, 
\end{align*}
which gives the second identity in \eqref{eq:b5}. 

Next, let us prove the first identity in \eqref{eq:b5}. First, applying \eqref{eq:c34} to $f=h$, we establish an analogue of \eqref{eq:c27}:
\begin{equation}
(D^{-1}H)h=Hh^{\rm int}-{\jap{h}}h^{\rm int}.
\label{eq:c30}
\end{equation}
Next, we apply $H$ to the formula \eqref{eq:c35} of $\phi_E$ and use the second Lyapunov identity in \eqref{eq:c1a}:
\begin{align*}
H\phi_E&
=
\frac1E\bigl(-Hh^{\rm int}-H(D^{-1}H)(H+E)^{-1}h\bigr)
\\
&=
\frac1E\bigl(-Hh^{\rm int}+(D^{-1}H)H(H+E)^{-1}h+\jap{(H+E)^{-1}h,h^{\rm int}}h^{\rm int}    \bigr)
\\
&=\frac1E\bigl( -Hh^{\rm int}+(D^{-1}H)h-E (D^{-1}H)(H+E)^{-1}h +a(E)h^{\rm int}
\bigr).
\end{align*}
Using \eqref{eq:c30}, this transforms as
\begin{align*}
H\phi_E
=
\frac1E\bigl( -{\jap{h}} h^{\rm int}-E (D^{-1}H)(H+E)^{-1}h +a(E)h^{\rm int}
\bigr)
=E\phi_E+\mySigma(E)h^{\rm int}, 
\end{align*}
which proves \eqref{eq:b5}.

(ii)
Suppose, to get a contradiction, that
\[
\theta_E=c\phi_E
\]
for some $c\not=0$. Then we also have 
\[
H\theta_E=c H\phi_E.
\]
Let us write the last condition using the expressions for $H\phi_E$ and $H\theta_E$ from part (i):
\[
E\theta_E+\mySigma(E)h
=
cE\phi_E+c\mySigma(E)h^{\rm int}.
\]
Using $\theta_E=c\phi_E$ again, this yields
\[
\mySigma(E)h=c\mySigma(E)h^{\rm int}.
\]
Since by assumption $E$ is not an eigenvalue of $H$, we can cancel $\mySigma(E)$, which brings us to $h=ch^{\rm int}$. Recalling that 
\[
h(t)=\int_0^\infty \ee^{-t\lambda}\dd\mu(\lambda)
\quad\text{ and }\quad
h^{\rm int}(t)=\int_0^\infty \ee^{-t\lambda}\frac{\dd\mu(\lambda)}{\lambda},
\]
we see that the measures $\dd\mu(\lambda)$ and $\dd\mu(\lambda)/\lambda$ must be collinear. This implies that $\mu$ must be supported at a single point, i.e. $\rank H=1$. 
The proof of Theorem~\ref{thm:b2} is complete. 
\qed


\begin{thebibliography}{10}


\bibitem{BE}
A. Erdelyi et al., 
\emph{Higher Transcendental Functions,} 
Bateman Manuscript Project, Vols. 1, 2, 
McGraw--Hill, New York, 1953.



\bibitem{GR}
I.~S.~Gradshteyn, I.~M.~Ryzhik, 
\emph{Table of Integrals, Series, and Products}, 
Seventh Edition,
Academic Press 2007.

\bibitem{Howland}
J.~Howland, 
\emph{Spectral theory of operators of Hankel type.} I, II.
Indiana Univ. Math. J. \textbf{41} (1992), no. 2, 409--426, 427--434.


\bibitem{KP}
S.~V.~Khrushchev, V.~V.~Peller, 
\emph{Moduli of Hankel operators, past and future,}
in Linear and Complex Analysis Problem Book. 
Lecture Notes in Math., 1043, pp. 92--97. 
Springer, 1984. 


\bibitem{Magnus}
W.~Magnus, 
\emph{On the spectrum of Hilbert’s matrix}, 
Amer. J. Math., \textbf{72} (1950), 699--704.



\bibitem{MPT}
A.~V.~Megretski\u{\i}, V.~V.~Peller, S.~R.~Treil’, 
\emph{The inverse spectral problem for self-adjoint Hankel operators,} 
Acta Math. \textbf{174} (1995), no. 2, 241--309.

\bibitem{Mehler}
F.~G.~Mehler, 
\emph{Ueber eine mit den Kugel- und Cylinderfunctionen verwandte Function und ihre Anwendung in der Theorie der Elektricit\"atsvertheilung,}
Math. Ann. \textbf{18}, no. 2 (1881), 161--194. 

\bibitem{DLMF}
NIST Digital Library of Mathematical Functions. https://dlmf.nist.gov/, Release 1.2.6 of 2026-03-15. F. W. J. Olver, A. B. Olde Daalhuis, D. W. Lozier, B. I. Schneider, R. F. Boisvert, C. W. Clark, B. R. Miller, B. V. Saunders, H. S. Cohl, and M. A. McClain, eds.

\bibitem{Nikolski}
N.~Nikolski, 
\emph{Hardy Spaces},
Cambridge University Press, 2019. 

\bibitem{Ober}
R.~J.~Ober, 
\emph{A note on a system-theoretic approach to a conjecture by Peller-Khrushchev: the general case}, 
IMA J. Math. Control Inform. \textbf{7} (1990), no. 1, 35--45.

\bibitem{Peller}
V.~V.~Peller, 
\emph{Hankel operators and their applications,}
Springer, New York, 2003.

\bibitem{PuTreil0}
A. Pushnitski, S. Treil,
\emph{Unbounded integral Hankel operators,}
Funct. Anal. Appl. \textbf{59} (2025), no. 3, 297--320.

\bibitem{PuTreil}
A.~Pushnitski, S.~Treil,
\emph{Inverse spectral problems for positive Hankel operators,}
to appear in Analysis \& PDE, arXiv:2503.22189. 


\bibitem{PuSobolev}
A.~Pushnitski, A.~Sobolev,
\emph{Hankel operators with band spectra and elliptic functions,}
Duke Math. J. 174, no. 4 (2025), 685--746.

\bibitem{RS1}
M.~Reed, B.~Simon, 
\emph{Methods of modern mathematical physics. I. Functional analysis.} 
Second edition. Academic Press, New York, 1980.

\bibitem{Ros}
M.~Rosenblum, 
\emph{On the Hilbert matrix, I, II,} 
Proc. Amer. Math. Soc., 9 (1958), 137--140, 581--585.

\bibitem{Shanker}
H.~Shanker,
\emph{An integral equation for Whittaker's confluent hypergeometric function,}
Proc. Cambridge Philos. Soc. \textbf{45} (1949), 482--483. 

\bibitem{Simon81}
B.~Simon,
\emph{Spectrum and continuum eigenfunctions of Schr\"odinger operators,}
J. Funct. Anal. \textbf{42} (1981), 347--355. 

\bibitem{Simon}
B.~Simon,
\emph{Schr\"odinger semigroups,}
Bulletin of the AMS, \textbf{7}(3) (1982), 447--526.

\bibitem{Shnol}
E.~E.~Shnol', 
\emph{On the behavior of the eigenfunctions of Schr\"odinger's equation} (in Russian),
Mat. Sb. (N.S.), \textbf{42(84)}:3 (1957), 273--286.


\bibitem{Titchmarsh}
E.~C.~Titchmarsh,
\emph{Eigenfunction expansions associated with second-order differential equations,}
Oxford University Press, Oxford, 1962.


\bibitem{Ya1}
D.~Yafaev,
\emph{A commutator method for the diagonalization of Hankel operators,}
Funct. Anal. Appl. \textbf{44} no. 4 (2010), 295--306.



\bibitem{Ya2}
D.~Yafaev,
\emph{Quasi-diagonalization of Hankel operators,}
Journal d'Analyse Mathematique, \textbf{133} (2017), 133--182.


\bibitem{Ya3}
D.~Yafaev,
\emph{Spectral and scattering theory for perturbations of the Carleman operator,}
St. Petersburg Math. J. \textbf{25} (2014), no. 2, 339--359.

\bibitem{Yakubovich}
S.~B.~Yakubovich, 
\emph{Index Transforms,} 
World Scientific, 1996.

\end{thebibliography}
\end{document}